\documentclass[12pt]{amsart}     
\usepackage{amsmath,amssymb,amsthm,mathrsfs}     
\usepackage{latexsym,ulem}    
\usepackage{hyperref}     
\usepackage{multirow}     
\usepackage{color,colordvi,graphicx}
\newtheorem{theorem}{Theorem}     
\numberwithin{theorem}{section}
\newtheorem{lemma}[theorem]{Lemma}     
\newtheorem{corollary}[theorem]{Corollary}     
\newtheorem{proposition}[theorem]{Proposition}     

\theoremstyle{definition}  
\newtheorem{definition}[theorem]{Definition}     
\newtheorem{example}[theorem]{Example}     
\newtheorem{remark}[theorem]{Remark}

\newcounter{FNC}[page]
\def\fauxfootnote#1{{\addtocounter{FNC}{2}$\Magenta{^\fnsymbol{FNC}}$%
     \let\thefootnote\relax\footnotetext{\Magenta{$^\fnsymbol{FNC}$#1}}}}
\renewcommand{\P}{{\mathbb P}}

\newcommand{\R}{{\mathbb R}}

\newcommand{\sz}{\footnotesize}
\newcommand{\sm}{\normalsize}

\newcommand{\calC}{{\mathcal C}}
\newcommand{\calJ}{{\mathcal J}}
\newcommand{\calS}{{\mathcal S}}

\newcommand{\lszero}[3]{C^{#1}_{#2}(#3)}
\newcommand{\lsone}[3]{C^{#1}_{\mbox{\footnotesize{\textup{st}}}(#2)}(#3)}

\newcommand{\frakm}{{\mathfrak m}}

\DeclareMathOperator{\st}{{\rm st}}
\DeclareMathOperator{\codim}{{\rm codim}}
\DeclareMathOperator{\reg}{{\rm reg}}
\DeclareMathOperator{\image}{{\rm image}}
\DeclareMathOperator{\HF}{{\it HF}}
\DeclareMathOperator{\HP}{{\it HP}}

\newcommand{\defcolor}[1]{{\color{blue}#1}}
\newcommand{\demph}[1]{\defcolor{{\it #1}}}

\title{Bivariate Semialgebraic Splines}
\author[M.~DiPasquale]{Michael DiPasquale}     
\address{Michael DiPasquale\\     
         Department of Mathematics\\     
         Colorado State University\\     
         Fort Collins\\
         CO \ 80521\\     
         USA}     
\email{michael.dipasquale@colostate.edu}     
\urladdr{\url{https://midipasq.github.io/}}   
\author[F.~Sottile]{Frank Sottile}     
\address{Frank Sottile\\     
         Department of Mathematics\\     
         Texas A\&M University\\     
         College Station\\     
         Texas \ 77843\\     
         USA}     
\email{sottile@math.tamu.edu}     
\urladdr{\url{http://www.math.tamu.edu/~sottile}}   

\thanks{Research of Sottile supported in part by NSF grant DMS-1501370.}

\begin{document}     
     
%%%%%%%%%%%%%%%%%%%%%%%%%%%%%%%%%%%%%%%%%%%%%%%%%%%%%%%%%%%%%%%%%%%%%%%%%%%%%%%%%
\begin{abstract}
 Semialgebraic splines are bivariate splines over meshes whose edges are arcs of algebraic curves.
 They were first considered by Wang, Chui, and Stiller.
 We compute the dimension of the space of semialgebraic splines in two extreme cases.
 If the polynomials defining the edges span a three-dimensional space of polynomials, then we compute the dimensions
 from the dimensions for a corresponding rectilinear mesh.
 If the mesh is sufficiently generic, we give a formula for the dimension of the spline space valid in large degree and
 bound how large the degree must be for the formula to hold.
 We also study the dimension of the spline space in examples which do not satisfy either extreme.
 The results are derived using commutative and homological algebra.
\end{abstract}       
%%%%%%%%%%%%%%%%%%%%%%%%%%%%%%%%%%%%%%%%%%%%%%%%%%%%%%%%%%%%%%%%%%%%%%%%%%%%%%%%%

\maketitle     
%%%%%%%%%%%%%%%%%%%%%%%%%%%%%%%%%%%%%%%%%%%%%%%%%%%%%%%%%%%%%%%%%%%%%%%%%%%% 
%     
\section{Introduction}\label{Sec:intro}    
A bivariate spline is a function on a domain in $\R^2$ that is piecewise a polynomial with respect
to a \demph{mesh}, that is, a cell decomposition $\Delta$ of the domain. 
A fundamental question is to determine the dimension of the vector space of splines on $\Delta$ with a 
given smoothness and whose polynomial constituents have at most a fixed degree.
Traditionally, $\Delta$ is \demph{rectilinear}, that is, a simplicial~\cite{Strang} or polygonal~\cite{Schum84} complex,
but it is  useful to consider splines where the edges of $\Delta$ are arcs of algebraic curves~\cite{Davydov16,DS17,DY17}. 
Such \demph{semialgebraic splines} were studied by Wang and Chui~\cite{ChuiWang83,Wang75,Wang85}, and 
by Stiller~\cite{Stiller83}.
 
When $\Delta$ is rectilinear, classical spline spaces were recast in terms of graded modules and
homological algebra by Billera~\cite{Billera88}, who  developed this with Rose~\cite{BiRo91,BiRo92}.  
Further foundational work by Schenck and Stillman was given in~\cite{ShSt97a,ShSt97b}.

With Sun, we observed that this homological machinery carries over to semialgebraic meshes~\cite{DiPSS}.
We used this to determined the dimensions of spline spaces for two distinct generalizations of rectilinear meshes when
there is a single interior vertex. 
In one, the forms underlying the edges span a two-dimensional vector space, allowing a direct comparison via tensor
product to the rectilinear case. 
In the other, generic, case, the forms have distinct tangents at the vertex and do not simultaneously vanish at any other
point, which corresponds to rectilinear edges having all slopes at the vertex distinct.

We consider two cases of semialgebraic meshes extending those from~\cite{DiPSS}.
For the first, we assume that every edge form lies in the vector space spanned by three forms of degree $n$
(called a \demph{net}).
This implies that the forms at every interior vertex $v$ span a two-dimensional space,
and that $\Delta$ is mapped to a rectilinear mesh $\phi(\Delta)$ under the rational map given by the net.
When the forms defining the net have no common zeros,
the space of semialgebraic splines is obtained from a spline space on $\phi(\Delta)$ via a tensor product.
We describe this in Section~\ref{S:Net} and illustrate how the Morgan-Scott phenomenon~\cite{MoSc77} occurs
for these spaces.
We determine the dimension of the spline space valid for large degree $d$ and bound how large $d$ must be for this
dimension formula to hold. 

The second case is of splines over generic semialgebraic meshes.
In addition to the conditions from~\cite{DiPSS} for genericity at interior vertices, we also assume a certain acyclicity
condition on $\Delta$ that appeared in~\cite{McSh09}.
These conditions are satisfied for almost all meshes $\Delta$ and thus by generic meshes in the
algebraic geometry sense.
For such generic meshes, we determine the dimension of the spline space valid for large degree $d$ and we modify results from~\cite{DiP_Mixed} to bound how
large $d$ must be for this dimension formula to hold.

In Section~\ref{S:SplineModules}, we fix our notation, give background on spline modules, and give examples.
Section~\ref{S:Net} studies splines whose meshes come from a net, relates them to rectilinear spline spaces, and studies the
Morgan-Scott phenomenon.
Section~\ref{S:generic} studies generic meshes, and in Section~\ref{S:regularity} we give a bound on when our
formula for the dimension of the spline space for a generic mesh holds.
In Section 6, we discuss meshes to which our theorems do not apply, but for which homological methods compute dimensions of
spline spaces.

%%%%%%%%%%%%%%%%%%%%%%%%%%%%%%%%%%%%%%%%%%%%%%%%%%%%%%%%%%%%%%%%%%%%%%%%%%%%%%%%%
\section{Modules of semialgebraic splines}\label{S:SplineModules}

Billera~\cite{Billera88} introduced methods from homological algebra into the study of splines.
This was refined by Billera and Rose~\cite{BiRo91,BiRo92} and by Schenck and Stillman~\cite{ShSt97a,ShSt97b}, who
viewed splines as a graded module over the polynomial ring, so that
the dimension of spline spaces is given by the Hilbert function of the module.
In~\cite{DiPSS} we observed that this homological approach remains valid for semialgebraic splines.
We recall that.
For more background, we recommend \S~8.3 of~\cite{CLO05} or~\cite{Eisenbud}.

A \demph{mesh} $\Delta$ is a finite cell complex in $\R^2$ whose 1-cells are arcs of irreducible real
algebraic curves. 
The 2-cells of $\Delta$ are \demph{faces}, the 1-cells \demph{edges}, and the 0-cells are \demph{vertices}. 
We assume that each vertex and edge of $\Delta$ lies in the boundary of some face (it is \demph{pure}), that it is
connected, and that it is \demph{hereditary}: for any faces $\sigma,\sigma'$ sharing a vertex $v$,
there is a sequence $\sigma=\sigma_0,\sigma_1,\dotsc,\sigma_n=\sigma'$ of faces containing $v$ such that   
each pair $\sigma_{i-1},\sigma_i$ for $i=1,\dotsc,n$ shares an edge.
Write $\defcolor{|\Delta|}\subset\R^2$ for the support of $\Delta$.
We assume that $|\Delta|$ is contractible and require that each connected component of the intersection of 
two cells of $\Delta$ is a cell of $\Delta$.
Write $\Delta^\circ_i$ for the set of $i$-cells of $\Delta$ that lie in the interior of $|\Delta|$.
Every face $\sigma$ of $\Delta$ inherits the orientation of $\R^2$ and we fix an orientation of
each edge $\tau\in\Delta^\circ_1$.

Figure~\ref{F:cell_complex} shows a mesh $\Delta$ with $|\Delta|=[-\sqrt{2},\sqrt{2}]\times[-1,1]$.
%%%%%%%%%%%%%%%%%%%%%%%%%%%%%%%%%%%%%%%%%%%%%%%%%%%%%%%%%%%%%%%%%%%%%%%%%%%%%%%%%
\begin{figure}[htb]

  \includegraphics{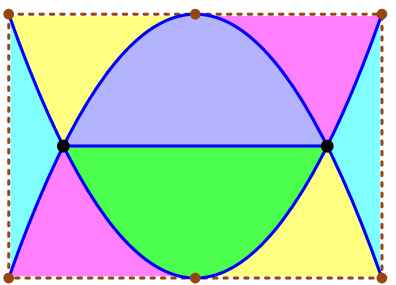}

 \caption{Mesh with two interior vertices.}
 \label{F:cell_complex}
\end{figure}
%%%%%%%%%%%%%%%%%%%%%%%%%%%%%%%%%%%%%%%%%%%%%%%%%%%%%%%%%%%%%%%%%%%%%%%%%%%%%%%%%
It has two interior vertices,  $(-1,0)$ and $(1,0)$, nine interior edges, and eight faces, 
and the curves underlying the edges are the parabolas
$y=x^2-1$, $y=1-x^2$ and the $x$-axis.

Let $S$ be a ring.
A \demph{chain complex $\calC$} is a sequence $C_0,C_1,\dotsc,C_n$ of $S$-modules with $S$-module maps
$\partial_i\colon C_i\to C_{i-1}$, whose compositions vanish, $\partial_{i-1}\circ\partial_i=0$,
so that the kernel of $\partial_{i-1}$ contains the image of $\partial_i$.
(Here, $C_{-1}=C_{n+1}=0$.)
The \demph{homology} of $\calC$ is the sequence of $S$-modules 
$\defcolor{H_i(\calC)}:=\ker(\partial_{i-1})/\image(\partial_i)$, for
$i=0,\dotsc,n$. 

Let \defcolor{$\calS[\Delta]$} be the chain complex whose $i$th module has a basis given by the cells of
$\Delta^\circ_i$ and whose maps are induced by the boundary maps on the cells.
For the mesh $\Delta$ of Figure~\ref{F:cell_complex}, $\calS[\Delta]$ is $S^8\to S^9\to S^2$.
Since the interior cells subdivide $|\Delta|$ with its boundary \defcolor{$\partial|\Delta|$} removed, the homology of the
complex $\calS[\Delta]$ is the relative  homology \defcolor{$H_i(|\Delta|,\partial|\Delta|;S)$}.
This vanishes when $i=0$.
As $|\Delta|$ is connected and contractible, we also have that $H_1(\calS[\Delta])=0$ and $H_2(\calS[\Delta])=S$.  

For integers $r,d\geq 0$, let $\defcolor{\widetilde{C}^r_d(\Delta)}$ be the real vector space of functions $f$
on $|\Delta|$ which have continuous $r$th order partial derivatives and whose
restriction to each face $\sigma$ of $\Delta$ is a polynomial $f_\sigma$ of degree at most $d$.
By~\cite{Wang75} (see also~\cite[Cor.~1.3]{BiRo92}), elements $f\in \widetilde{C}^r_d(\Delta)$ are lists
$(f_\sigma\mid\sigma\in\Delta_2)$ of polynomials such that if $\tau\in\Delta^\circ_1$ is an interior edge with
defining equation $g_\tau(x,y)=0$ that borders the faces $\sigma,\sigma'$, then   
$g_\tau^{r+1}$ divides the difference $f_\sigma-f_{\sigma'}$.
(The quotient is the smoothing cofactor at $\tau$.)

Figure~\ref{F:spline_graphs} displays the graphs of two splines on the mesh $\Delta$ of 
%%%%%%%%%%%%%%%%%%%%%%%%%%%%%%%%%%%%%%%%%%%%%%%%%%%%%%%%%%%%%%%%%%%%%%%%%%%%%%%%%
\begin{figure}[htb]

\raisebox{10pt}{\includegraphics[height=160pt]{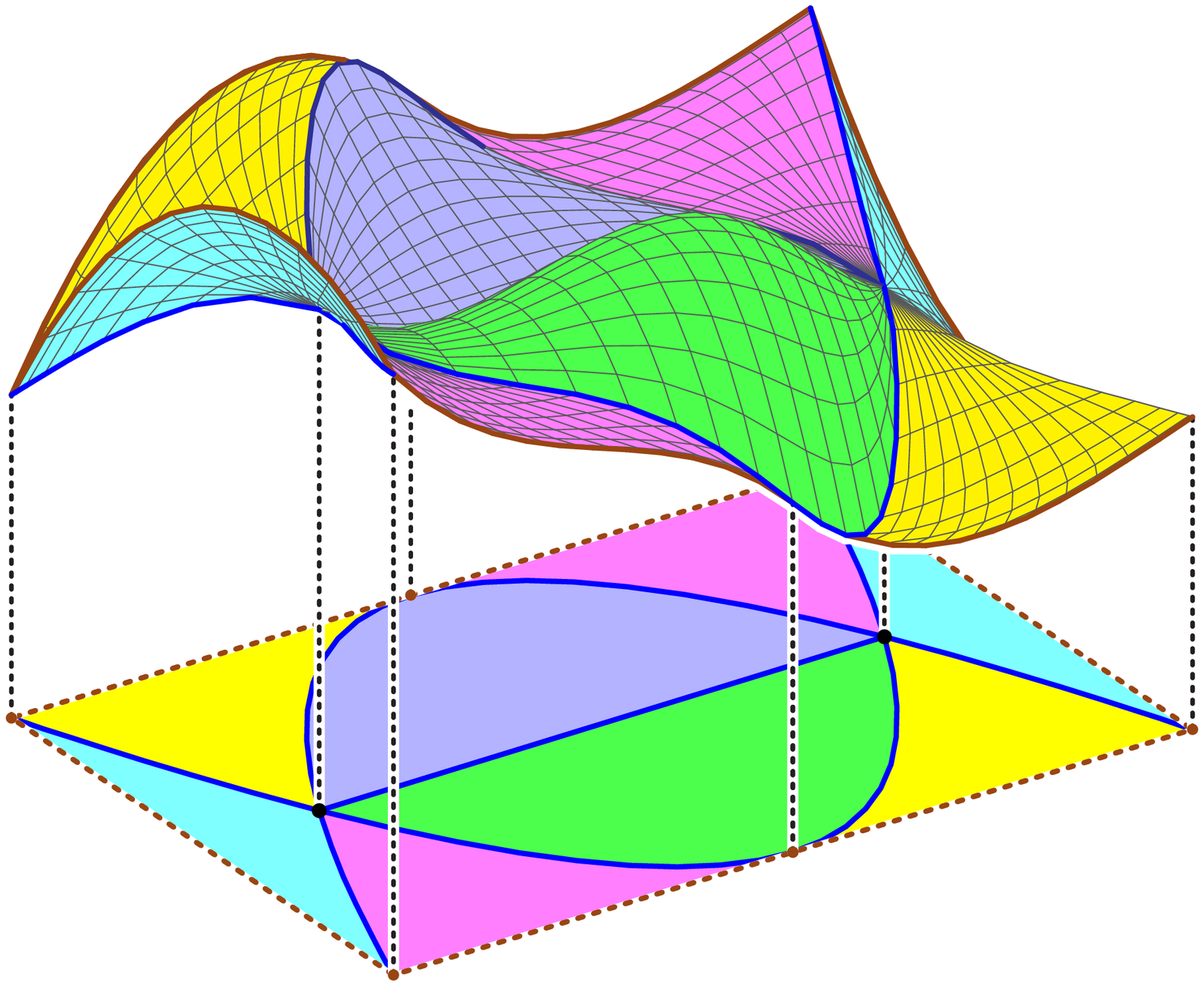}}
\qquad
\includegraphics[height=180pt]{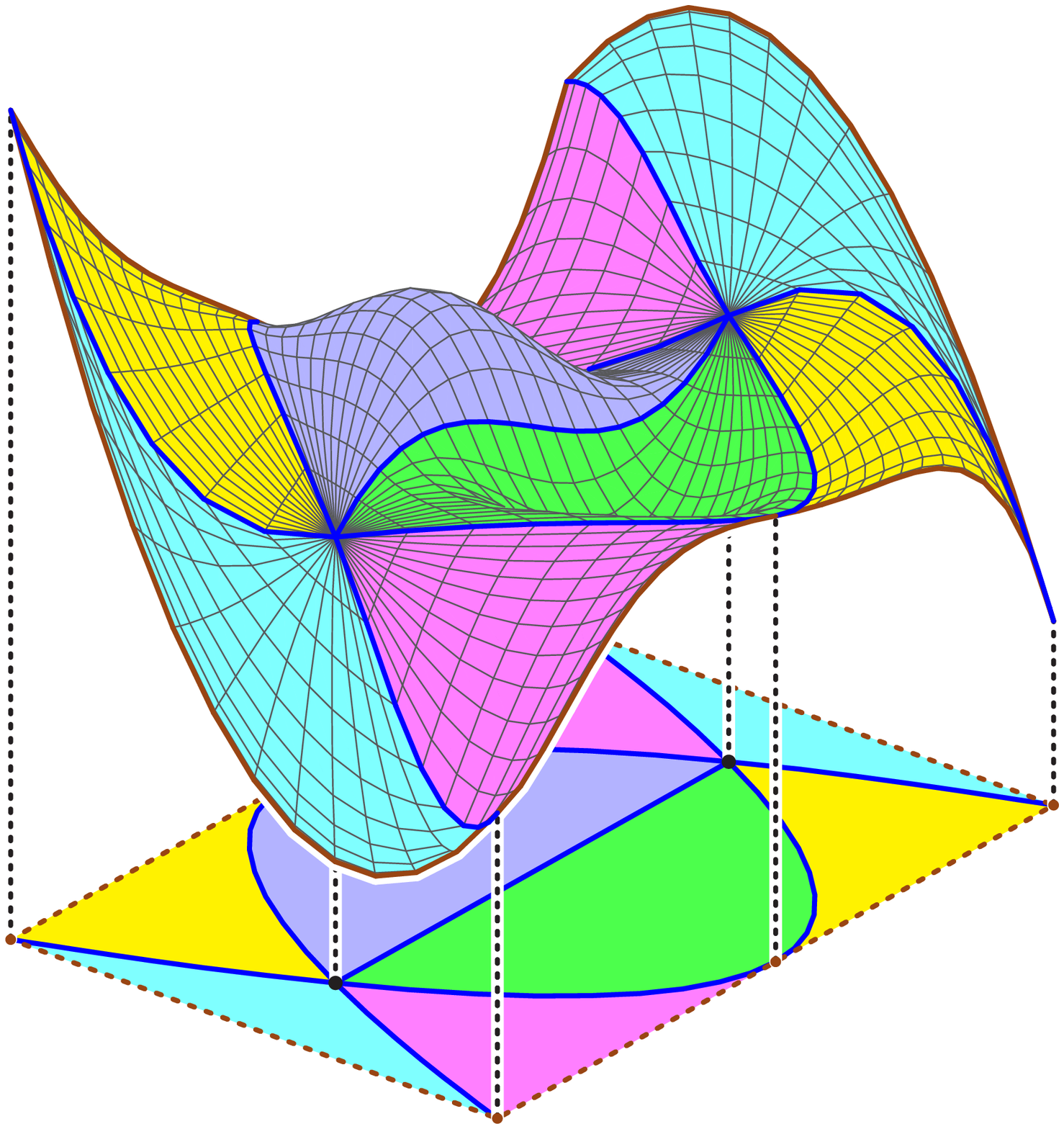}

\caption{Graphs of splines.}
\label{F:spline_graphs}
\end{figure}
%%%%%%%%%%%%%%%%%%%%%%%%%%%%%%%%%%%%%%%%%%%%%%%%%%%%%%%%%%%%%%%%%%%%%%%%%%%%%%%%%
Figure~\ref{F:cell_complex}.
The spline on the left lies in $\widetilde{C}^0_3(\Delta)$ and that on the right lies in
$\widetilde{C}^1_5(\Delta)$.

Billera and Rose~\cite{BiRo91} observed that homogenizing spline spaces enables a global homological
approach to compute them. 
Let $\defcolor{S}:=\R[x,y,z]$ be the homogeneous coordinate ring of $\P^2(\R)$.
Homogeneous polynomials are called \demph{forms}.
Write $\langle G_1,\dotsc,G_t\rangle$ for the ideal of $S$ generated by forms
$G_1,\dotsc,G_t\in S$.
Let $\defcolor{C^r_d(\Delta)}$ be the vector space of lists $(F_\sigma \mid \sigma\in\Delta_2)$ of
forms in $S$ of degree $d$ such that if $f_\sigma:= F_\sigma(x,y,1)$ is the dehomogenization of
$F_\sigma$, then $(f_\sigma \mid \sigma\in\Delta_2)\in \widetilde{C}^r_d(\Delta)$.  
Likewise let $G_\tau$ be the homogenization of the polynomial $g_\tau$ defining the curve underlying $\tau$.
We call $G_\tau$ the \demph{edge form} of $\tau$.

The \demph{spline module} $\defcolor{C^r(\Delta)}:=\bigoplus_d C^r_d(\Delta)$ is the direct sum of these homogenized spline
spaces. 
It is a graded module of the graded ring $S$.

%%%%%%%%%%%%%%%%%%%%%%%%%%%%%%%%%%%%%%%%%%%%%%%%%%%%%%%%%%%%%%%%%%%%%%%%%%%%%%%%%
\begin{lemma}\label{L:SplineModule}
 The spline module $C^r(\Delta)$ is finitely generated.
 It is the kernel of the map
 \begin{equation}\label{Eq:spline_kernel}
   S^{\Delta_2}\ \simeq\ \bigoplus_{\sigma\in\Delta_2} S\ \xrightarrow{\ \ \partial\ \ }
   \bigoplus_{\tau\in\Delta^\circ_1} S/\langle G_\tau^{r+1}\rangle\ ,
 \end{equation}
 where if $F=(F_\sigma \mid \sigma\in\Delta_2)\in S^{\Delta_2}$ 
 and $\tau\in\Delta^\circ_1$, then the $\tau$-component of $\partial F$ is the
 difference $F_\sigma-F_{\sigma'}$, where $\tau$ is a component of the intersection $\sigma\cap\sigma'$
 and its orientation agrees with that induced from $\sigma$,
 but is opposite to that induced from $\sigma'$.
\end{lemma}
%%%%%%%%%%%%%%%%%%%%%%%%%%%%%%%%%%%%%%%%%%%%%%%%%%%%%%%%%%%%%%%%%%%%%%%%%%%%%%%%%

\begin{remark}
 We perform all computations in this paper by implementing the map~\eqref{Eq:spline_kernel} in the computer algebra system
 Macaulay2~\cite{M2}.
 They are available from the website accompanying this
 article\footnote{{\tt https://www.math.tamu.edu/\~{}sottile/research/stories/bivariate/index.html}}.
\end{remark}

Let $M=\bigoplus_d M_d$ be a finitely generated graded $S$-module.
The \demph{Hilbert function} of $M$ records the dimensions of its graded pieces, 
$\defcolor{\HF(M,d)}:=\dim_\R M_d$.
There is an integer $d_0\geq 0$ such that if $d>d_0$, then the Hilbert function is a polynomial, called the
\demph{Hilbert polynomial} of $M$, \defcolor{$\HP(M,d)$}~\cite{Eisenbud}.
The \demph{postulation number} of $M$ is the minimal such $d_0$, the greatest integer at which the Hilbert
function and Hilbert polynomial disagree. 
The reason for these definitions is that the problem of computing the dimensions $\dim C^r_d(\Delta)$ of the
spline spaces is equivalent to computing the Hilbert function of the spline module $C^r(\Delta)$, which equals its
Hilbert polynomial for $d>d_0$.

Table~\ref{T:HilbertFunction} gives the Hilbert function and Hilbert polynomial of $C^r(\Delta)$ for
$r=0,\dotsc,3$, where $\Delta$ is the mesh of Figure~\ref{F:cell_complex}.
The Hilbert polynomials will be explained in Section~\ref{S:ConcludingRemarks}.
%%%%%%%%%%%%%%%%%%%%%%%%%%%%%%%%%%%%%%%%%%%%%%%%%%%%%%%%%%%%%%%%%%%%%%%%%%%%%%%%%
\begin{table}[htb]
 \caption{Hilbert function and polynomial of $C^r(\Delta)$.}
 \label{T:HilbertFunction}
  \begin{tabular}{|r|rrrrrrrrrrrrrr|c|r|}
    \hline  
    ${r\backslash d}$\rule{0pt}{12pt}
       &0&1&2&3&4&  5&  6&  7&  8&   9&  10&  11&  12&  13&Polynomial\\\hline
    $0$&1&\underline{3}&\defcolor{{\bf 11}} & 26& 49 & 80& 119& 166& 221& 284& 355& 434& 521& 616&\raisebox{-6pt}{\rule{0pt}{18pt}}%
                         $4d^2{-}5d{+}5$\\
    $1$&1&3&6&10&\defcolor{{\bf 19}}&\underline{34}&57& 87& 125&  171& 225& 287& 357& 435&\raisebox{-6pt}{\rule{0pt}{18pt}}%
                         $4d^2{-}22d{+}45$\\
    $2$&1&3&6&10&15& 21& \defcolor{{\bf 32}}& 48& 71&\underline{102}&  140&  185& 238& 299&\raisebox{-6pt}{\rule{0pt}{18pt}}%
                         $4d^2{-}39d{+}130$\\
    $3$&1&3&6&10&15& 21& 28& 36& \defcolor{{\bf 49}}&  67&  90&  120& 159&\underline{205}&\raisebox{-6pt}{\rule{0pt}{18pt}}%
                      %134 162 193
                         $4d^2{-}56d{+}258$ \\\hline
% \binom{d+2}{2}   1, 3, 6, 10, 15, 21, 28, 36, 45, 55, 66, 78, 91, 105, 120, 136, 153
%    3*d^2/2 - (10*r+1)*d/2 + (1,9,28,57,97,147,208,279,361,453,556,669,793,927)
%                              0 1  2  3  4   5   6   7   8   9  10  11  12  13
% dim C^r_d(\Delta)  3*d^2/2 - (10*r+1)*d/2 + c(r), where c(r) is:
% r even  c(r) = 21*r^2/4 + 3*r +1:
% r odd   c(r) = 21*r^2/4 + 3*r + 3/4:
   &1&3&6&10&15&21&28&36&45&55&66&78&91&105&\raisebox{-6pt}{\rule{0pt}{18pt}}%
     $\frac{1}{2}d^2+\frac{3}{2}d+1$\\\hline
 \end{tabular}
\end{table}
%%%%%%%%%%%%%%%%%%%%%%%%%%%%%%%%%%%%%%%%%%%%%%%%%%%%%%%%%%%%%%%%%%%%%%%%%%%%%%%%%
Its last row is the Hilbert function/polynomial of $S$; these are 
splines that are restrictions of polynomials on $\R^2$.
The values at the postulation numbers are underlined, they are at degrees 1, 5, 9, and 13.
The first space containing a spline that is not the restriction of a polynomial is \defcolor{{\bf highlighted}}.

For $\tau\in\Delta^\circ_1$, let $\defcolor{J(\tau)}:=\langle G_\tau^{r+1}\rangle$ be the principal ideal
generated by $G_\tau^{r+1}$ and for $v\in\Delta^\circ_0$, let \defcolor{$J(v)$} be the
ideal generated by all $J(\tau)$ where $\tau$ is incident on $v$.
Let $\calJ_1$ and $\calJ_0$ be the direct sums of these ideals, 
\[
     \calJ_1\ :=\ \bigoplus_{\tau\in\Delta^\circ_1} J(\tau)
     \qquad \mbox{and}\qquad    
     \calJ_0\ :=\ \bigoplus_{v\in\Delta^\circ_0} J(v)\,.
\]
Then $\calJ \colon\calJ_1\xrightarrow{\partial_1}\calJ_0$ is a complex of $S$-modules, with $\partial_1$
the obvious map.
This is a subcomplex of the chain complex $\defcolor{\calS[\Delta]}$ that computes the relative homology
$H_*(|\Delta|,\partial|\Delta|;S)$.
We have the short exact sequence of complexes of $S$-modules,
 \begin{equation}\label{Eq:SES}
   0\ \longrightarrow\ \calJ\ \longrightarrow\ 
   \calS\ \longrightarrow\ \calS/\calJ\ \longrightarrow\  0\,,
 \end{equation}
where the quotient $\calS/\calJ$ is Billera-Schenck-Stillman complex,
 \[
  0\ \longrightarrow\ 
  \bigoplus_{\sigma\in\Delta_2} S\ \xrightarrow{\ \partial_2\ }\ 
  \bigoplus_{\tau\in\Delta^\circ_1} S/J(\tau)\ \xrightarrow{\ \partial_1\ }\ 
  \bigoplus_{v\in\Delta^\circ_0} S/J(v)\ \longrightarrow\ 0\,.
 \]
\begin{remark}\label{rem:CellComplexNotation}
When it is necessary to indicate the underlying mesh $\Delta$ for $\calJ,\calS,$ and $\calS/\calJ$,
we will write $\defcolor{\calJ[\Delta]},\defcolor{\calS[\Delta]}$, and $\defcolor{\calS/\calJ[\Delta]}$.\hfill$\diamond$
\end{remark}
Observe that $C^r(\Delta)$ is the kernel of $\partial_2$.
That is, $C^r(\Delta)=H_2(\calS/\calJ)$.
The short exact sequence~\eqref{Eq:SES} gives the long exact sequence in homology
(note that $H_2(\calJ)=0$).
 \begin{multline*}
   \qquad \;0\to H_2(\calS)\to H_2(\calS/\calJ)\to H_1(\calJ)\to 
    H_1(\calS) \\
   \to H_1(\calS/\calJ)\to H_0(\calJ)\to H_0(\calS)\to H_0(\calS/\calJ)\to 0\,.\ \qquad
 \end{multline*}
 Since $H_{*}(\calS)$ is the homology of $\Delta$ relative to its boundary and $|\Delta|$ is contractible, we have
 that $H_1(\calS)=H_0(\calS)=H_0(\calS/\calJ)=0$.
 This gives the following.
 
%%%%%%%%%%%%%%%%%%%%%%%%%%%%%%%%%%%%%%%%%%%%%%%%%%%%%%%%%%%%%%%%%%%%%%%%%%%%%%%%%
\begin{proposition}\label{P:exact_sequences}
 $H_1(\calS/\calJ)\simeq H_0(\calJ)$ and $C^r(\Delta)\simeq S\oplus H_1(\calJ)$, with the factor of $S$ corresponding to
 the  splines that are restrictions of polynomials. 
\end{proposition}
%%%%%%%%%%%%%%%%%%%%%%%%%%%%%%%%%%%%%%%%%%%%%%%%%%%%%%%%%%%%%%%%%%%%%%%%%%%%%%%%%

%%%%%%%%%%%%%%%%%%%%%%%%%%%%%%%%%%%%%%%%%%%%%%%%%%%%%%%%%%%%%%%%%%%%%%%%%%%%%%%%

Write \defcolor{$\phi_2$} for the number of faces of $\Delta$, \defcolor{$\phi_1$} for the number of interior edges,
and \defcolor{$\phi_0$} for the number of interior vertices, and for an interior edge $\tau\in\Delta^\circ_1$, let
\defcolor{$n_\tau$} be the degree of the edge form $G_\tau$ of $\tau$.

%%%%%%%%%%%%%%%%%%%%%%%%%%%%%%%%%%%%%%%%%%%%%%%%%%%%%%%%%%%%%%%%%%%%%%%%%%%%%%%%%
\begin{corollary}\label{C:DimFormula}
 Suppose that the support $|\Delta|$ of $\Delta$ is contractible.
 Then for $r$ and $d$,
 \begin{equation}\label{Eq:dimFormula}
   \dim C^r_d(\Delta)\ =\ 
    (\phi_2-\phi_1)\tbinom{d+2}{2} + \sum_{\tau\in\Delta^\circ_1} \tbinom{d-(r+1)n_\tau+2}{2}
     + \sum_{v\in\Delta^\circ_0}\dim(S/J(v))_d
     + \dim H_0(\calJ)_d\,.
 \end{equation}
 For $d\gg 0$, $\dim(S/J(v))_d$ is the degree of the scheme defined by $J(v)$. 
\end{corollary}
%%%%%%%%%%%%%%%%%%%%%%%%%%%%%%%%%%%%%%%%%%%%%%%%%%%%%%%%%%%%%%%%%%%%%%%%%%%%%%%%%

\begin{remark}\label{rem:Multiplicity}
If $I\subset S$ is a homogeneous ideal so that $\HP(S/I,d)=\mu$ for some $\mu\in\mathbb{Z}_{>0}$ (equivalently $\dim
(S/I)_d=\mu$ for $d\gg 0$), then $\mu$ is called the \demph{multiplicity} of $I$. \hfill$\diamond$
\end{remark}

%%%%%%%%%%%%%%%%%%%%%%%%%%%%%%%%%%%%%%%%%%%%%%%%%%%%%%%%%%%%%%%%%%%%%%%%%%%%%%%%%

%%%%%%%%%%%%%%%%%%%%%%%%%%%%%%%%%%%%%%%%%%%%%%%%%%%%%%%%%%%%%%%%%%%%%%%%%%%%%%%%%
%
\section{Nets}\label{S:Net}

In a rectilinear mesh, the edges lie along lines whose equations lie in the
three-dimension\-al space of degree one polynomials.
We consider semialgebraic meshes whose edge forms are similarly constrained to lie in a three-dimen\-sion\-al space of
homogeneous forms, which we call a \demph{net}.
In such a mesh, there is a unique edge between two general points.
Also, the edges incident to a vertex form a pencil, which is a case treated in~\cite{DiPSS}.

%%%%%%%%%%%%%%%%%%%%%%%%%%%%%%%%%%%%%%%%%%%%%%%%%%%%%%%%%%%%%%%%%%%%%%%%%%%%%%%%%
\begin{example}\label{Ex:MS}
 The three homogeneous quadratic polynomials $\{x^2 - yz , y^2 - xz , z^2 + xy \}$ span a net.
 Let $(1,1)$, $(2,1)$, $(1,2)$, $(\frac{3}{11},2)$, $(2,\frac{3}{11})$, and $(\frac{5}{2},\frac{5}{2})$ be six points.
 Each pair lies on a unique curve from this net.
 Figure~\ref{F:Qnet} shows a mesh with these vertices.
%%%%%%%%%%%%%%%%%%%%%%%%%%%%%%%%%%%%%%%%%%%%%%%%%%%%%%%%%%%%%%%%%%%%%%%%%%%%%%%%
 \begin{figure}[htb]
 \centering   
 \begin{picture}(426,160)(-16,-5)
  \put(0,-5){\includegraphics{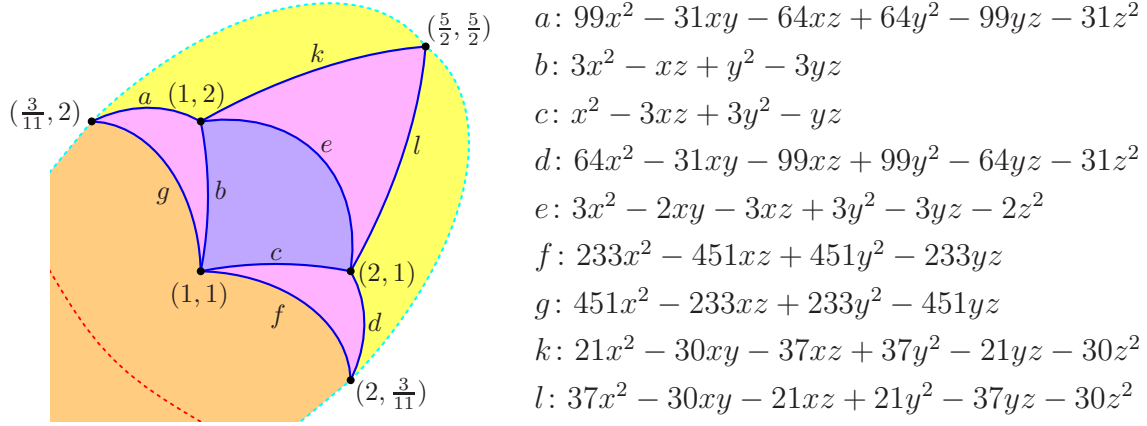}}
              \put(142,141){\sz$(\frac{5}{2},\frac{5}{2})$}     
   \put(-16,109){\sz$(\frac{3}{11}, 2)$}   \put(44,116){\sz$(1, 2)$}  
   \put( 45,40){\sz$(1, 1)$}   \put(116,48){\sz$(2,1)$}  
       \put(115,3){\sz$(2, \frac{3}{11})$}
   \put(98,131){\sz$k$}   \put(137,96){\sz$l$}
   \put(33,115){\sz$a$}    \put(102,97){\sz$e$}
   \put(40,79){\sz$g$}    \put(62,79){\sz$b$} \put(83,56){\sz$c$}
   \put(83,32){\sz$f$}    \put(120,29){\sz$d$}
   \put(183,144){\sm$a\colon 99x^2-31xy-64xz+64y^2-99yz-31z^2$}
   \put(183,126){\sm$b\colon 3x^2-xz+y^2-3yz$}
   \put(183,108){\sm$c\colon x^2-3xz+3y^2-yz$}
   \put(183, 90){\sm$d\colon 64x^2-31xy-99xz+99y^2-64yz-31z^2$}
   \put(183, 72){\sm$e\colon 3x^2-2xy-3xz+3y^2-3yz-2z^2$}
   \put(183,54){\sm$f\colon 233x^2-451xz+451y^2-233yz$}
   \put(183,36){\sm$g\colon 451x^2-233xz+233y^2-451yz$}
   \put(183,18){\sm$k\colon 21x^2-30xy-37xz+37y^2-21yz-30z^2$}
   \put(183,0){\sm$l\colon 37x^2-30xy-21xz+21y^2-37yz-30z^2$}
 \end{picture}
 \caption{A mesh from a quadratic net.}
 \label{F:Qnet}
\end{figure} 
%%%%%%%%%%%%%%%%%%%%%%%%%%%%%%%%%%%%%%%%%%%%%%%%%%%%%%%%%%%%%%%%%%%%%%%%%%%%%%%%%
It has nine interior edges labeled $a,\dotsc,l$ and three boundary edges, all from the net.
Only a portion of the lower left face is shown.
We also list the forms underlying the interior edges.\hfill$\diamond$
\end{example}
%%%%%%%%%%%%%%%%%%%%%%%%%%%%%%%%%%%%%%%%%%%%%%%%%%%%%%%%%%%%%%%%%%%%%%%%%%%%%%%%%

Let $f,g,h\in S$ be  forms of the same degree and assume that their linear span \defcolor{$N$}
is three-dimensional.
The net $N$ defines a rational map $\defcolor{\phi_N}\colon \P^2\, -\to \P^2$ which is given by
$\phi_N[x:y:z]=[f(x,y,z):g(x,y,z):h(x,y,z)]$ at points $[x:y:z]$ that are not common zeroes of $f,g,h$.
Write $\defcolor{R}:=\R[u,v,w]$ for the homogeneous coordinate ring of the codomain of $\phi_N$.
The pullback map $\defcolor{\phi^*_N}\colon R\to S$ is defined by $\phi^*_N(u)=f$, $\phi^*_N(v)=g,$ and $\phi^*_N(w)=h$.
When the net $N$ is clear we write $\phi$ and $\phi^*$ for $\phi_N$ and $\phi^*_N$.

Suppose that a mesh $\Delta$ has edge forms from the net spanned by $f,g,h$.
For an interior edge $\tau$ of $\Delta$, $\phi(\tau)$ is a line segment defined by the linear form
\defcolor{$\ell_{\phi(\tau)}$} such that $\phi^*(\ell_{\phi(\tau)})=G_\tau$.
General rational maps from $\P^2$ to $\P^2$ may be quite complicated.
We will only consider rational maps which have no \demph{basepoints}.
That is, the polynomials $f$, $g$, and $h$ have no common zeros in $\P^2$, which implies that they form a
\demph{regular sequence} in $S$.

Even when $\phi$ has no basepoints, $\phi(\Delta)$ may not be a traditional rectilinear mesh, since the rational map
may not be injective on $\Delta$.
For example, the map $\phi$ from the net of Example~\ref{Ex:MS} is not injective on lower left face of the mesh
$\Delta$ in Figure~\ref{F:Qnet}---the Jacobian of $\phi$ vanishes on the cubic $x^3+y^3+5xy-1$, an arc of which is shown.
The map $\phi$ `folds'  the plane along this cubic.
Figure~\ref{F:RI} shows the the image $\phi(\Delta)$, using 
the generators $f$, $g$, and $h=100z^2+100xy$ of the net from Example~\ref{Ex:MS}.
The red curve is the image of the cubic of Figure~\ref{F:Qnet}.
It bounds the image of the lower left face of $\Delta$.
%%%%%%%%%%%%%%%%%%%%%%%%%%%%%%%%%%%%%%%%%%%%%%%%%%%%%%%%%%%%%%%%%%%%%%%%%%%%%%%%%
\begin{figure}[htb]
 \centering
 \begin{picture}(120,120) 
   \put(0,0){\includegraphics{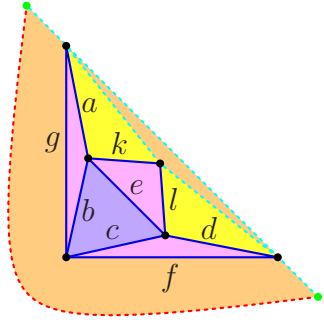}}
   \put(29,78){\sm$a$}     \put(40,62){\sm$k$}
   \put(15,65){\sm$g$}   \put(29,37){\sm$b$}   \put(47,47){\sm$e$} \put(62,41){\sm$l$}
   \put(38,30){\sm$c$}   \put(74,30){\sm$d$}
   \put(59,12){\sm$f$}
 \end{picture}
 \caption{Image of quadratic net from Example~\ref{Ex:MS}.}
 \label{F:RI}
\end{figure}  
%%%%%%%%%%%%%%%%%%%%%%%%%%%%%%%%%%%%%%%%%%%%%%%%%%%%%%%%%%%%%%%%%%%%%%%%%%%%%%%%%
The straight line boundary is the image of the arc bounding that lower-left face, and it folds over itself at its
endpoints. 
The images of the edges are labeled $a,\dotsc,l$.

Despite the potentially complicated geometry of $\phi(\Delta)$, the algebra is clear.
Define
\[
  \defcolor{C^r(\phi(\Delta))}\ :=\ \ker\Bigl(
  \bigoplus_{\phi(\sigma)\in\phi(\Delta)_2} R\ \xrightarrow{\ \ \delta_2\ \ }\
  \bigoplus_{\phi(\tau)\in\phi(\Delta)^\circ_1} \frac{R}{J(\phi(\tau))}\Bigr)
\]
as in Lemma~\ref{L:SplineModule}.

Let us recall a standard repackaging of the Hilbert function (see Exercise 10.12 in~\cite{Eisenbud}).
The \demph{Hilbert series} of $C^r(\Delta)$ is the generating series 
\[
  \defcolor{HS(C^r(\Delta),t)}\ :=\ \sum_{t\ge 0} \dim C^r_d(\Delta) t^d
\]
of the Hilbert function.
It is a rational function  of the form $\frac{p(t)}{(1-t)^3}$ for some polynomial $p(t)$ with integer
coefficients. 

%%%%%%%%%%%%%%%%%%%%%%%%%%%%%%%%%%%%%%%%%%%%%%%%%%%%%%%%%%%%%%%%%%%%%%%%%%%%%%%%%
\begin{theorem}\label{T:NSubdivisionStructure}
  Let $\Delta$ be a mesh whose edge forms lie in the net spanned by base point free forms $f,g,h\in S$ of degree $n$.
  Let  $\phi_N$ and $\phi^*_N$ be the associated maps.
 As $S$-modules, $C^r(\Delta)\cong C^r(\phi(\Delta))\otimes_R S.$  As (left) $R$-modules, 
\[
  C^r(\Delta)\ \cong\ C^r(\phi(\Delta))\otimes_{\R} \dfrac{S}{\langle f,g,h\rangle}\,.
\]
 Consequently, 
\begin{enumerate}
\item $\displaystyle{ \dim C^r_d(\Delta)\ =\
         \sum_{ni+j=d} \dim C^r_i(\phi(\Delta))\cdot\dim \Bigl( \frac{S}{\langle f,g,h\rangle} \Bigr)_{\!j}}\,$, and 
\item if $HS(C^r(\phi(\Delta)),t)=\dfrac{p(t)}{(1-t)^3}$, then $HS(C^r(\Delta),t)=\dfrac{p(t^n)}{(1-t)^3}$.
\end{enumerate}
\end{theorem}
%%%%%%%%%%%%%%%%%%%%%%%%%%%%%%%%%%%%%%%%%%%%%%%%%%%%%%%%%%%%%%%%%%%%%%%%%%%%%%%%%

%%%%%%%%%%%%%%%%%%%%%%%%%%%%%%%%%%%%%%%%%%%%%%%%%%%%%%%%%%%%%%%%%%%%%%%%%%%%%%%%%
\begin{proof}
  As $f,g,h$ form a regular sequence, the pullback map $\phi^*_N\colon R\to S$ is flat~\cite{Ha}.
  Thus extending scalars from $R$ to $S$, $M\mapsto M\otimes_RS$, preserves kernels and cokernels of $R$-module maps.
   For instance, $J(\phi(\tau))\otimes_R S\cong J(\tau)$ and $\frac{R}{J(\phi(\tau))}\otimes_R S\cong\frac{S}{J(\tau)}$.

 The $R$-module $C^r(\phi(\Delta))$ is the kernel of the map
\[
  \psi\ \colon\ R^{\phi(\Delta)_2}\ \xrightarrow{\ \delta_2\ }\
   \bigoplus_{\phi(\tau)\in\phi(\Delta)^\circ_1} \frac{R}{J(\phi(\tau))}\,.
\]
Tensoring with $S$ yields (via flatness) the map, 
\[
  \psi\otimes_R S\ \colon\ S^{\Delta_2}\ \xrightarrow{\ \delta_2\ }\
  \bigoplus_{\tau\in\Delta^\circ_1}\frac{S}{J(\tau)}\,,
\]
whose kernel is $C^r(\Delta)$.
Since ${-}\otimes_R S$ preserves kernels, $C^r(\Delta)\cong C^r(\phi(\Delta))\otimes_R S$.

Finitely generated flat graded $R$-modules are free~\cite[Cor.~6.6]{Eisenbud}, so $S$ is a free $R$-module.
It is minimally generated as an $R$-module by $S/\frakm_R S=S/\langle f,g,h\rangle$, where $\frakm_R=\langle
u,v,w \rangle$.
Thus, as left $R$-modules, $S\cong R\otimes_\R S/\langle f,g,h\rangle$.
So, as left $R$-modules,
\[
  C^r(\Delta)\ \cong\  C^r(\phi(\Delta))\otimes_R S\ \cong\
  (C^r(\phi(\Delta))\otimes_R R)\otimes_\R \frac{S}{\langle f,g,h\rangle}\ \cong\
  C^r(\phi(\Delta))\otimes_{\R} \frac{S}{\langle f,g,h\rangle}\,.
\]
The statements concerning $\dim C^r_d(\Delta)$ and $HS(C^r(\Delta),t)$ follow from these tensor product
decompositions. 
\end{proof}
%%%%%%%%%%%%%%%%%%%%%%%%%%%%%%%%%%%%%%%%%%%%%%%%%%%%%%%%%%%%%%%%%%%%%%%%%%%%%%%%%

Reading off the Hilbert polynomial of $C^r(\Delta)$ from Theorem~\ref{T:NSubdivisionStructure} is not difficult but it is
technical.
Set $\defcolor{c_n(j)}:=\dim(S/\langle f,g,h\rangle)_j$.
Since $f,g,h$ form a regular sequence and each has degree $n$, a Hilbert series computation shows that
$c_j(n)$ is the coefficient of $t^j$ in the expansion of $(1+t+t^2+\cdots+t^{n-1})^3$.

%%%%%%%%%%%%%%%%%%%%%%%%%%%%%%%%%%%%%%%%%%%%%%%%%%%%%%%%%%%%%%%%%%%%%%%%%%%%%%%%%
\begin{corollary}\label{C:Nets}
  Let $\Delta$ be a mesh with edge forms from the net $N$ spanned by base point free forms $f,g,h\in S$ of degree $n$, with 
   associated maps $\phi_N$ and $\phi^*_N$.
  Suppose that $\HP(C^r(\phi(\Delta)),d)=ad^2+bd+c$, where $a,b,c$ are constants.
  Then $\HP(C^r(\Delta),d)=Ad^2+Bd+C$, where
 \begin{eqnarray*}
   A & =&a\,,\\
   B & =& bn-2\dfrac{n^2-1+2c_n(2n)}{n}a\,,\ \mbox{ and}\\
   C & =& (n^2-1+3c_n(2n))a-(n^2-1+c_n(2n))b+n^2c\, .
 \end{eqnarray*}
 The postulation number for $\dim C^r_d(\Delta)$ is at most $n(d_0+3)-3$, where 
  $d_0$ is the postulation number for $\dim C_d^r(\phi(\Delta))$.
\end{corollary}
%%%%%%%%%%%%%%%%%%%%%%%%%%%%%%%%%%%%%%%%%%%%%%%%%%%%%%%%%%%%%%%%%%%%%%%%%%%%%%%%%

That is, if $\dim C^r_d(\phi(\Delta))=\HP(C^r(\phi(\Delta)),d)$ for $d> d_0$, then when  $d>n(d_0+3)-3$, 
we have $\dim C^r_d(\Delta)=\HP(C^r(\Delta),d)$.

%%%%%%%%%%%%%%%%%%%%%%%%%%%%%%%%%%%%%%%%%%%%%%%%%%%%%%%%%%%%%%%%%%%%%%%%%%%%%%%%%
\begin{proof}
  From Theorem~\ref{T:NSubdivisionStructure}(1), for all $i\gg 0$ , $\HP(C^r(\Delta),ni)$ equals 
  \[
    \HP(C^r(\phi(\Delta)),i)\ +\ c_n(n)\HP(C^r(\phi(\Delta)),i{-}1)\ +\ c_n(2n)\HP(C^r(\phi(\Delta)),i{-}2)\,.
 \]
 Since this is an identity among polynomials, it holds for all $i$.
 Expanding both sides as polynomials in $i$ and equating coefficients yields
 \begin{eqnarray*}
   An^2 & =&a(1+c_n(n)+c_n(2n))\,,\\
   Bn & =& b(1+c_n(n)+c_n(2n))-2a(c_n(n)+2c_n(2n))\,,\ \mbox{ and}\\
   C & =& a(c_n(n)+4c_n(2n))-b(c_n(n)+2c_n(2n))+c(1+c_n(n)+c_n(2n))\,.
\end{eqnarray*}
The equations for $A,B,C$ in terms of $a,b,c$ follow from the identity $1+c_n(n)+c_n(2n)=n^2$
of Lemma~\ref{L:CNsums}.

Now suppose $\dim C^r_d(\phi(\Delta))=\HP(C^r(\phi(\Delta)),d)$ for $d>d_0$.
Then $\dim C^r_d(\Delta)=\HP(C^r(\Delta),d)$ if $\dim C^r_i(\phi(\Delta))$ for $i\le d_0$ does not contribute to
$\dim C^r_d(\Delta)$.
Since $c_n(j)=0$ for $j>3n-3$, Theorem~\ref{T:NSubdivisionStructure} implies that this happens when $d>nd_0+3n-3$, or
$d>n(d_0+3)-3$.
This completes the proof.
\end{proof}
%%%%%%%%%%%%%%%%%%%%%%%%%%%%%%%%%%%%%%%%%%%%%%%%%%%%%%%%%%%%%%%%%%%%%%%%%%%%%%%%%

%%%%%%%%%%%%%%%%%%%%%%%%%%%%%%%%%%%%%%%%%%%%%%%%%%%%%%%%%%%%%%%%%%%%%%%%%%%%%%%%%
\begin{lemma}\label{L:CNsums}
 For $0\leq i<n$, we have $c_n(i)+c_n(i{+}n)+c_n(i{+}2n)=n^2$.
\end{lemma}
%%%%%%%%%%%%%%%%%%%%%%%%%%%%%%%%%%%%%%%%%%%%%%%%%%%%%%%%%%%%%%%%%%%%%%%%%%%%%%%%%

%%%%%%%%%%%%%%%%%%%%%%%%%%%%%%%%%%%%%%%%%%%%%%%%%%%%%%%%%%%%%%%%%%%%%%%%%%%%%%%%%
\begin{proof}
  If $\zeta$ is an $n$th root of unity, then $a\zeta^{i+n}=a\zeta^i$.
  When  $\zeta\neq 1$,  it is a root of
  \[
    \defcolor{p(t)}\ :=\ \sum_{i=0}^{n-1}\bigl(c_n(i)+c_n(i{+}n)+c_n(i{+}2n)\bigr)t^i\,.
  \]
  as $p(\zeta)=\sum_{i=0}^{3n-3} c_n(i)\zeta^i=(1+\zeta+\zeta^2+\dotsb+\zeta^{n-1})^3=0$.
  Thus $p(t)$ is a constant multiple of $1+t+\dotsb+t^{n-1}$ and so $c_n(i)+c_n(i{+}n)+c_n(i{+}2n)$ is a constant.
  Noting that $p(1)=(1+\dotsb+1)^3=n^3$ completes the proof.
\end{proof}
%%%%%%%%%%%%%%%%%%%%%%%%%%%%%%%%%%%%%%%%%%%%%%%%%%%%%%%%%%%%%%%%%%%%%%%%%%%%%%%%%

%%%%%%%%%%%%%%%%%%%%%%%%%%%%%%%%%%%%%%%%%%%%%%%%%%%%%%%%%%%%%%%%%%%%%%%%%%%%%%%%%
\begin{example}\label{E:MorganScottNet}
  Let $\phi$ be the rational map for the net $N$ of Example~\ref{Ex:MS}.
  If $\Delta$ is a mesh whose edge forms come from $N$, then $\phi(\Delta)$ is a rectilinear mesh.
  Using Corollary~\ref{C:Nets}, if $\HP(C^r(\phi(\Delta)),d)=ad^2+bd+c$ then 
 \begin{equation}\label{E:NetHilbert}
    \HP(C^r(\Delta),d)\ =\ ad^2+(2b-3a)d+(3a-3b+4c)\,.
  \end{equation}

  Suppose that \defcolor{$\Delta$} is the mesh from Example~\ref{Ex:MS}.
  Then $\phi(\Delta)$ is the mesh with the seven triangles shown in Figure~\ref{F:RI}.
  (The algebra defining $C^r(\phi(\Delta))$ does not see the lower left lune-shaped region.)
  This is the projection of an octahedron to $\R^2$, and the three lines connecting the images of antipodal points of the 
  octahedron intersect in a common point.
  We show this in Figure~\ref{F:NMS}.
%%%%%%%%%%%%%%%%%%%%%%%%%%%%%%%%%%%%%%%%%%%%%%%%%%%%%%%%%%%%%%%%%%%%%%%%%%%%%%%%%
  \begin{figure}[htb]
   \centering
    \includegraphics{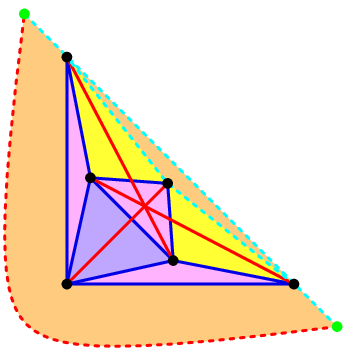}  \qquad\qquad
    \includegraphics{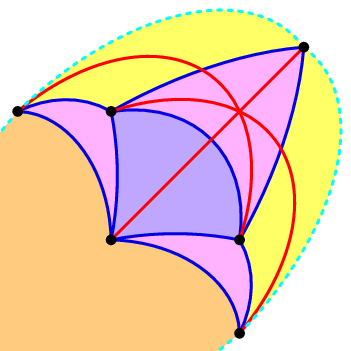}
   \caption{Lines in $\phi(\Delta)$ and their preimage curves in $\Delta$.}
   \label{F:NMS}
  \end{figure}  
%%%%%%%%%%%%%%%%%%%%%%%%%%%%%%%%%%%%%%%%%%%%%%%%%%%%%%%%%%%%%%%%%%%%%%%%%%%%%%%%%
  As shown by Morgan and Scott~\cite{MoSc77}, such an octahedral configuration illustrates that the space of quadratic
  $C^1$ splines on a rectilinear mesh is sensitive to geometry.
  Namely, there is an `unexpected' spline in $C^1_2(\phi(\Delta))$ due to the three  coincident lines.
  This propagates to unexpected splines in $C^1_d(\Delta)$ for $d=4,5,6,7$.

  We explain this.
  Let $\Delta'$ be a mesh whose edges come from the net $N$, with the same topology as $\Delta$, but with some of its
  vertices perturbed so that the curves connecting opposite vertices are not coincident, and thus 
  $\phi(\Delta')$ does not satisfy the Morgan-Scott intersection property.
  Let $\defcolor{P_\phi(d)}:=\HP(C^1(\phi(\Delta)),d)=\HP(C^1(\phi(\Delta')),d)=\frac{7}{2}d^2-\frac{15}{2}d+7$.
  (We made this computation in Macaulay2~\cite{M2}; it also follows from~\cite[Cor.~4.5]{ShSt97a} and coincides with
  Schumaker's lower bound for $d\ge 4$~\cite{AS90}.)

 Let $\defcolor{P(d)}:=\HP(C^1(\Delta),d)=\HP(C^1(\Delta'),d)$.
 By~\eqref{E:NetHilbert}, $P(d)=\frac{7}{2}d^2-\frac{51}{2}d+61$.
 Since $\dim C^1_d(\phi(\Delta'))$ agrees with $P_\phi(d)$ for $d>d_0=0$, by Corollary~\ref{C:Nets} $\dim C^1_d(\Delta')$
 agrees with $P(d)$ for $d>n(d_0+3)-3=3$.
 On the other hand, $\dim C^1_d(\Delta)$ agrees with $P_\phi(d)$ for $d>d_0=2$, so $\dim C^1_d(\Delta)$ agrees with $P(d)$
 for $d>n(d_0+3)-3=7$.
 Table~\ref{T:MorganScottNet} displays values of the Hilbert functions and Hilbert polynomials.
 %%%%%%%%%%%%%%%%%%%%%%%%%%%%%%%%%%%%%%%%%%%%%%%%%%%%%%%%%%%%%%%%%%%%%%%%%%%%%%%%%
 \begin{table}
  \caption{Hilbert functions for Morgan-Scott nets in Example~\ref{E:MorganScottNet}}\label{T:MorganScottNet}
  \renewcommand{\arraystretch}{1.2}
  \begin{tabular}{|c|ccccccc|c|}\hline
  $d$ & 0 & 1 & 2 & 3 & 4 & 5 & 6 & $P_\phi(d)$ \\ \hline 
  $\dim C^1_d(\phi(\Delta'))$ & \underline{1} & 3 & 6 & 16 & 33 & 57 & 88 & $\frac{7}{2}d^2-\frac{15}{2}d+7$ \\  \hline
  $\dim C^1_d(\phi(\Delta))$ & 1 & 3 &\defcolor{\underline{{\bf 7}}}& 16 & 33 & 57 & 88 & $\frac{7}{2}d^2-\frac{15}{2}d+7$\\\hline
 \end{tabular}

 \vspace{10 pt}

 \begin{tabular}{|c|ccccccccccc|c|}\hline
   $d$ & 0 & 1 & 2 & 3 & 4 & 5 & 6 & 7 & 8 & 9 & 10 & $P(d)$\\\hline
   $\dim C^1_d(\Delta')$ & 1 & 3 & 6 & \underline{10} & 15 & 21 & 34 & 54 & 81 & 115 & 156 &
          $\frac{7}{2}d^2-\frac{51}{2}d+61$\\\hline 
   $\dim C^1_d(\Delta)$ & 1 & 3 & 6 & 10 &\defcolor{{\bf 16}} & \defcolor{{\bf 24}} & 
         \defcolor{{\bf 37}} &\defcolor{\underline{{\bf 55}}}& 81 & 115 & 156 & $\frac{7}{2}d^2-\frac{51}{2}d+61$\\\hline
 \end{tabular}
 \end{table}
%%%%%%%%%%%%%%%%%%%%%%%%%%%%%%%%%%%%%%%%%%%%%%%%%%%%%%%%%%%%%%%%%%%%%%%%%%%%%%%%%
 The difference of $1$ between $\dim C^1_2(\phi(\Delta'))$ and $\dim C^1_2(\phi(\Delta))$ results in
 differences of $1,3,3,1$ in degrees $d=4,5,6,7$ between $\dim C^1_d(\Delta')$ and $\dim C^1_d(\Delta)$.
 These differences are the coefficients of $(1+t)^3$.
 All of these observations are explained by Theorem~\ref{T:NSubdivisionStructure}.\hfill$\diamond$
\end{example}
%%%%%%%%%%%%%%%%%%%%%%%%%%%%%%%%%%%%%%%%%%%%%%%%%%%%%%%%%%%%%%%%%%%%%%%%%%%%%%%%%

%%%%%%%%%%%%%%%%%%%%%%%%%%%%%%%%%%%%%%%%%%%%%%%%%%%%%%%%%%%%%%%%%%%%%%%%%%%%%%%%%
\section{Generic meshes}\label{S:generic}

We determine the Hilbert polynomial of $C^r(\Delta)$ for a generic mesh $\Delta$.  
We explain exactly what we mean by a generic mesh in Definition~\ref{D:generic}, which
extends conditions from~\cite{DiPSS} by an acyclicity condition on $\Delta$.  
Our arguments follow the those in~\cite{ShSt97a,ShSt97b} with modifications reflecting the more complicated
geometry of the mesh as in~\cite{DiP_Assoc,McSh09}. 

We first define open subsets of $\Delta$ that encode the behavior of the chain complex $\calS/\calJ$ under localization 
at a prime ideal $P\subset S$ and are used in our characterization of generic meshes.
Two faces $\sigma,\sigma'$ of $\Delta$ are \demph{$P$-adjacent} if they are adjacent in $\Delta$ and if there is an edge 
$\tau\in\sigma\cap\sigma'$ between them whose edge form $G_\tau$ lies in $P$.
That is, $P$ is either the ideal of the curve underlying $\tau$, or it is the ideal of a point lying along this curve, or
it is the ideal of two complex conjugate points through which the curve passes. 

Let \defcolor{$\sim_P$} be the equivalence relation on faces of $\Delta$ generated by $P$-adjacency.
Then $\sigma\sim_P\sigma'$ if and only if there are sequences $\sigma=\sigma_0,\sigma_1,\dotsc,\sigma_k=\sigma'$ of
faces and edges $\tau_1,\dotsc,\tau_k$ such that for $i=1,\ldots,k$, $\tau_i$ is between $\sigma_{i-1}$ and $\sigma_i$ and 
$G_{\tau_i}\in P$.
Write \defcolor{$[\sigma]_P$} for the equivalence class of $\sigma$ under $\sim_P$, and use this to define a subset 
\defcolor{$\Delta_{P,\sigma}$} of $\Delta$ whose faces are $[\sigma]_P$, and
whose edges and vertices are as follows.
\begin{itemize}
\item edges $\tau$ of $\Delta_{P,\sigma}$ are interior edges $\tau$ of $\Delta$ with $G_\tau\in P$ and which lie along a face
  in $[\sigma]_P$. 
\item vertices $v$ of $\Delta_{P,\sigma}$ are interior vertices $v$ of $\Delta$ such that every interior edge $\tau$
  incident on $v$ has $G_\tau\in P$ and $v$ lies on a face in $[\sigma]_P$.
\end{itemize}

The union of the cells in $\Delta_{P,\sigma}$ is a connected open subset of $|\Delta|$.
We also write $\Delta_{P,\sigma}$ for this union, \defcolor{$\partial\Delta_{P,\sigma}$} for its
boundary, and \defcolor{$\overline{\Delta_{P,\sigma}}$} for the closure of $\Delta_{P,\sigma}$.
Note that  $\partial\Delta_{P,\sigma}$ consists of those edges $\tau$ and vertices
$v$ in $\overline{\Delta_{P,\sigma}}$ with $G_\tau\notin P$ and $J(v)\not\subset P$, respectively.

%%%%%%%%%%%%%%%%%%%%%%%%%%%%%%%%%%%%%%%%%%%%%%%%%%%%%%%%%%%%%%%%%%%%%%%%%%%%%%%%%
\begin{example}\label{E:TopologicalCondition}
  Consider the mesh $\Delta$ from Figure~\ref{F:cell_complex} whose faces are labeled as in 
  Figure~\ref{F:Topological_Condition}.

  Let $P=\langle yz+x^2-z^2 \rangle$ be the ideal of the downward facing parabola.
  There are four equivalence classes of $P$-adjacent faces,
  $\{\sigma_1,\sigma_2,\sigma_3\}$, $\{\sigma_4,\sigma_5\}$, $\{\sigma_6,\sigma_7\}$ and $\{\sigma_8\}$.
%%%%%%%%%%%%%%%%%%%%%%%%%%%%%%%%%%%%%%%%%%%%%%%%%%%%%%%%%%%%%%%%%%%%%%%%%%%%%%%%%
\begin{figure}[htb]
\centering
\begin{picture}(112,80)
 \put(0,0){\includegraphics{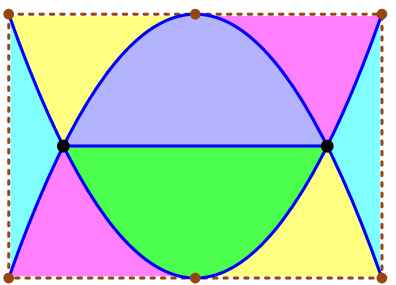}}
 \put(51,60){$\sigma_2$}   
 \put(51,19){$\sigma_8$}   
 \put(15,67){$\sigma_1$}    \put(85,67){$\sigma_3$}   
 \put(15,12){$\sigma_7$}    \put(85,12){$\sigma_5$}   
 \put(98,37){$\sigma_4$}   
 \put( 4,37){$\sigma_6$}   
\end{picture}

\caption{Mesh for Example~\ref{E:GenericityAssumptions}.}\label{F:Topological_Condition}	
\end{figure}
%%%%%%%%%%%%%%%%%%%%%%%%%%%%%%%%%%%%%%%%%%%%%%%%%%%%%%%%%%%%%%%%%%%%%%%%%%%%%%%%%
Then $\Delta_{P,\sigma_1}$ consists of the faces $\sigma_1,\sigma_2,$ and
$\sigma_3$, along with the two edges between them.
Also, $\Delta_{P,\sigma_4}$ consists of the faces $\sigma_4,\sigma_5$ and the edge between them.
Likewise, $\Delta_{P,\sigma_6}$ consists of the faces $\sigma_6,\sigma_7$ and the edge between them.
Finally, $\Delta_{P,\sigma_8}$ consists of the single face $\sigma_8$.
All subsets $\Delta_{P,\sigma}$ are contractible.

Now let $P=\langle x+z,y\rangle$ be the prime ideal of the left interior vertex.
Then all faces are equivalent as every edge $\tau$ lies along a curve incident to this vertex.
Thus $\Delta_{P,\sigma_1}$ contains every interior cell of $\Delta$ and it is contractible.

Let \defcolor{$Q$} be the prime ideal $\langle x,z \rangle$ which defines the point at infinity where the two parabolas
meet.
Again, all faces are equivalent under $Q$-adjacency.
However, $\Delta_{Q,\sigma_1}$ does not contain all interior cells of $\Delta$.
Neither the horizontal edge between the two interior vertices $(-1,0)$ and $(1,0)$ nor those vertices lie in
$\Delta_{Q,\sigma_1}$.
Thus $\Delta_{Q,\sigma_1}$ is not contractible.

If $P$ is a prime ideal defining a curve that is neither a parabola nor line from $\Delta$ nor
a prime ideal of a point on one of these parabolas or lines, then $\Delta_{P,\sigma_i}$ consists only of
$\sigma_i$ for all $i=1,\ldots,8$, and it is again contractible.\hfill$\diamond$
\end{example}
%%%%%%%%%%%%%%%%%%%%%%%%%%%%%%%%%%%%%%%%%%%%%%%%%%%%%%%%%%%%%%%%%%%%%%%%%%%%%%%%%

For any prime ideal $P\subset S$, we define the complex
 \[ %\ \begin{equation}\label{E:S}
  \defcolor{\calS[\Delta_{P,\sigma}]} \ \colon\ 
  \bigoplus_{\sigma'\in [\sigma]_P} S\  \xrightarrow{\ \partial_2\ }\
  \bigoplus_{\tau\in\Delta_{P,\sigma}} S\  \xrightarrow{\ \partial_1\ }\
  \bigoplus_{v\in \Delta_{P,\sigma}} S\,.
 \]%  \end{equation}
This is the cellular chain complex computing the homology of $\overline{\Delta_{P,\sigma}}$ relative to $\partial \Delta_{P,\sigma}$ so that
$H_i(\calS[\Delta_{P,\sigma}])\cong H_i(\overline{\Delta_{P,\sigma}},\partial\Delta_{P,\sigma};S)$.
We also have the chain complexes
 \[ 
  \defcolor{\calJ[\Delta_{P,\sigma}]}\ \colon\ 
  \bigoplus_{\tau\in\Delta_{P,\sigma}} J(\tau)\ \xrightarrow{\ \partial_1\ }\
  \bigoplus_{\tau\in\Delta_{P,\sigma}} J(v)
 \]
and
 \[ 
   \defcolor{\calS/\calJ[\Delta_{P,\sigma}]}\ \colon\
   \bigoplus_{\sigma'\in [\sigma]_P} S\ \xrightarrow{\ \partial_2\ }\
   \bigoplus_{\tau\in\Delta_{P,\sigma}}\frac{S}{J(\tau)}\  \xrightarrow{\ \partial_1\ }\
   \bigoplus_{v\in \Delta_{P,\sigma}} \frac{S}{J(v)}\,.
 \]
These fit into the short exact sequence of chain complexes, 
 \begin{equation}\label{E:JSSmodJ}
  0\ \rightarrow\  \calJ[\Delta_{P,\sigma}]\ \rightarrow \ \calS[\Delta_{P,\sigma}]\
  \rightarrow\ \calS/\calJ[\Delta_{P,\sigma}]\ \rightarrow\ 0\,.
 \end{equation}
%

%%%%%%%%%%%%%%%%%%%%%%%%%%%%%%%%%%%%%%%%%%%%%%%%%%%%%%%%%%%%%%%%%%%%%%%%%%%%%%%%%
\begin{proposition}\label{prop:ChainLocalization}
 For a prime ideal $P\subset S$ there is an isomorphism of chain complexes
\[
  \left(\calS/\calJ[\Delta]\right)_P\ \cong\
  \bigoplus_{\sigma\in\Sigma_P} \left(\calS/\calJ[\Delta_{P,\sigma}]\right)_P\, ,
\]
where $\defcolor{\Sigma_P}\subset\Delta_2$ is a set of representatives of $\sim_P$ equivalence classes.
\end{proposition}
%%%%%%%%%%%%%%%%%%%%%%%%%%%%%%%%%%%%%%%%%%%%%%%%%%%%%%%%%%%%%%%%%%%%%%%%%%%%%%%%%

%%%%%%%%%%%%%%%%%%%%%%%%%%%%%%%%%%%%%%%%%%%%%%%%%%%%%%%%%%%%%%%%%%%%%%%%%%%%%%%%%
\begin{proof}
  If $\tau\in\Delta^\circ_1$, then $(S/J(\tau))_P=0$ unless $G_\tau\in P$.
  Likewise, if $v\in\Delta^\circ_0$, then $(S/J(v))_P=0$ unless $G_\tau\in P$ for every $\tau$ having endpoint $v$.
  The localization assertion follows from these two observations.
\end{proof}
%%%%%%%%%%%%%%%%%%%%%%%%%%%%%%%%%%%%%%%%%%%%%%%%%%%%%%%%%%%%%%%%%%%%%%%%%%%%%%%%%

We give a vanishing lemma for the homology of the chain complex $\calS[\Delta_{P,\sigma}]$.

%%%%%%%%%%%%%%%%%%%%%%%%%%%%%%%%%%%%%%%%%%%%%%%%%%%%%%%%%%%%%%%%%%%%%%%%%%%%%%%%%
\begin{lemma}\label{L:Contractible}
  We have that $H_1(\calS[\Delta_{P,\sigma}])=0$ if and only if $\Delta_{P,\sigma}$ is contractible.
\end{lemma}
%%%%%%%%%%%%%%%%%%%%%%%%%%%%%%%%%%%%%%%%%%%%%%%%%%%%%%%%%%%%%%%%%%%%%%%%%%%%%%%%%
%%%%%%%%%%%%%%%%%%%%%%%%%%%%%%%%%%%%%%%%%%%%%%%%%%%%%%%%%%%%%%%%%%%%%%%%%%%%%%%%%
\begin{proof}
From the observations above,
$H_*(\calS[\Delta_{P,\sigma}])\cong H_*(\overline{\Delta_{P,\sigma}},\partial\Delta_{P,\sigma};S)$ is the 
homology  with coefficients in $S$ of $\overline{\Delta_{P,\sigma}}$ relative to $\partial\Delta_{P,\sigma}$.
By Excision~\cite[Prop.~2.22]{Ha02}, this is isomorphic to the reduced homology of the
topological quotient $\defcolor{X}:=\overline{\Delta_{P,\sigma}}/\partial\Delta_{P,\sigma}$. 
	
Let $\defcolor{x}$ be the image of $\partial\Delta_{P,\sigma}$ in $X$.
Since $\Delta_{P,\sigma}$ is open and $\partial\Delta_{P,\sigma}$ is collapsed when $X$ is formed, $\Delta_{P,\sigma}$ is
homeomorphic to $X\smallsetminus x$.
	
Suppose $\partial\Delta_{P,\sigma}$ has $c$ components.
Then $X$ is homeomorphic to a $2$-sphere with $c-1$ points identified to the single point $x$.
Consequently, $X\smallsetminus x$ has a deformation retraction to a wedge of $c-1$ circles.
These circles are generators of both the 
homology $H_1(\Delta_{P,\sigma};S)$ and $H_1(X;S)$.
Since $H_1(X;S)\cong H_1(\calS[\Delta_{P,\sigma}])$, we see that $H_1(\calS[\Delta_{P,\sigma}])=0$ if and only if
$H_1(\Delta_{P,\sigma})=0$. 
Since $\Delta_{P,\sigma}$ is connected, it is contractible if and only if its first homology group vanishes, so we are done.
\end{proof}
%%%%%%%%%%%%%%%%%%%%%%%%%%%%%%%%%%%%%%%%%%%%%%%%%%%%%%%%%%%%%%%%%%%%%%%%%%%%%%%%%

\begin{remark}\label{rem:H0}
  By the proof of Lemma~\ref{L:Contractible}, $H_0(\calS[\partial\Delta_{P,\sigma}])=\widetilde{H}_0(X;S)$, the reduced
  homology of $X$ with coefficients in $S$. 
  Since $X$ is connected, this homology module vanishes.
\end{remark}

%%%%%%%%%%%%%%%%%%%%%%%%%%%%%%%%%%%%%%%%%%%%%%%%%%%%%%%%%%%%%%%%%%%%%%%%%%%%%%%%%
\begin{definition}\label{D:generic}
   A mesh $\Delta$ is \demph{generic} if it satisfies the following conditions.
\begin{enumerate}
  \item For every pair $\tau,\tau'$ of edges meeting at a vertex $v$, either $G_\tau=G_{\tau'}$ or the
    tangent lines of $\tau$ and $\tau'$ at $v$ are distinct. 
 \item For every vertex $v$, the radical of the ideal $J(v)$ is the ideal $I(v)$ of $v$.
 \item For every prime ideal $P\subset S$ and face $\sigma$ of $\Delta$, $\Delta_{P,\sigma}$ is contractible.
 \item For every face $\sigma$ of $\Delta$, the edge forms $\{G_\tau\mid \tau\subset\sigma\}$ are distinct.
\end{enumerate}
\end{definition}
%%%%%%%%%%%%%%%%%%%%%%%%%%%%%%%%%%%%%%%%%%%%%%%%%%%%%%%%%%%%%%%%%%%%%%%%%%%%%%%%%

 Our results do not require all the conditions of Definition~\ref{D:generic} (indeed Condition (4) is only required in
 Section~\ref{S:regularity}).
 We point out which conditions are necessary for each result.

\begin{example}\label{E:GenericityAssumptions}
  Let $\Delta$ be the mesh from Figures~\ref{F:cell_complex} and~\ref{F:Topological_Condition}.
  This satisfies Condition (1) of Definition~\ref{D:generic}.
  It fails Condition (4) since either of the faces $\sigma_2$ and $\sigma_8$ (see Figure~\ref{F:Topological_Condition})
  have two different edges lying on the same curve.  

  It also fails Condition (2).
  Indeed, let $v$ be the interior vertex $(-1,0)$ of $\Delta$.
  Then $J(v)=\langle (yz+x^2-z^2)^{r+1},(yz-x^2+z^2)^{r+1},y^{r+1}\rangle$.
  Its radical is $I=\langle y,x-z\rangle\cap \langle y,x+z\rangle$, which is supported at both interior vertices of
  $\Delta$.

  For (3), we saw in Example~\ref{E:TopologicalCondition} that if $Q=\langle x,z\rangle$ then $\Delta_{Q,\sigma_1}$ has the
  topological type of an open annulus so is not contractible. 
\end{example}

%%%%%%%%%%%%%%%%%%%%%%%%%%%%%%%%%%%%%%%%%%%%%%%%%%%%%%%%%%%%%%%%%%%%%%%%%%%%%%%%%
\begin{proposition}\label{prop:FiniteLength}
If $\Delta$ is generic, then $H_1(\calS/\calJ)\cong H_0(\calJ)$ has finite length.
\end{proposition}
%%%%%%%%%%%%%%%%%%%%%%%%%%%%%%%%%%%%%%%%%%%%%%%%%%%%%%%%%%%%%%%%%%%%%%%%%%%%%%%%%

If the mesh is not generic, then $H_0(J)$ need not be finite length, as Example~\ref{E:FirstMesh} illustrates.

%%%%%%%%%%%%%%%%%%%%%%%%%%%%%%%%%%%%%%%%%%%%%%%%%%%%%%%%%%%%%%%%%%%%%%%%%%%%%%%%%
\begin{proof}
  It is enough to show that $H_1(\calS/\calJ)_P=0$ for every prime ideal $P$ which is not the graded maximal ideal of $S$.
  We use that localization commutes with taking homology, so we can first localize the chain complex $\calS/\calJ$ and then
  take homology afterward.
  By Proposition~\ref{prop:ChainLocalization}, it is enough to show that $H_1(\calS/\calJ[\Delta_{P,\sigma}])=0$ for
  every prime ideal $P$ and face $\sigma$.
  We show this is implied by the contractability of $\Delta_{P,\sigma}$.

  Suppose $\Delta_{P,\sigma}$ consists of the interior of an isolated face.
  Then $\calS[\Delta_{P,\sigma}]_1=0$, so certainly $H_1(\calS/\calJ[\Delta_{P,\sigma}])=0$.

  Now suppose that $\Delta_{P,\sigma}$ does not contain any interior vertices of $\Delta$.
  Then, in the short exact sequence of complexes~\eqref{E:JSSmodJ},
  $\calJ[\Delta_{P,\sigma}]$ has length one and the other complexes have length two.
  The tail end of the long exact sequence in homology gives the surjection
\[
    H_1(\calS[\Delta_{P,\sigma}])\ \twoheadrightarrow\  H_1(\calS/\calJ[\Delta_{P,\sigma}])\,.
\]
  Since $\Delta$ satisfies Condition (3) of Definition~\ref{D:generic}, $\Delta_{P,\sigma}$ is contractible.  By Lemma~\ref{L:Contractible}, $H_1(\calS[\Delta_{P,\sigma}])=0$.  Hence $H_1(\calS/\calJ[\Delta_{P,\sigma}])=0$ as well, and we are done.

  Now suppose that $\Delta_{P,\sigma}$ contains an interior vertex $v$ of $\Delta$. 
  Since $P$ is not the maximal ideal and $\Delta$ satisfies Condition (2) of Definition~\ref{D:generic}, we must have $P=I(v)$, the
  ideal of the vertex $v$.  Hence $v$ is the only interior vertex of $\Delta_{P,\sigma}$.
  Invoking Condition (3) of Definition~\ref{D:generic} and Lemma~\ref{L:Contractible}, we again have $H_1(\calS[\Delta_{P,\sigma}])=0$.
  Furthermore, $H_0(\calS[\Delta_{P,\sigma}])=0$ by Remark~\ref{rem:H0}.
  
  Thus from the long exact sequence in homology induced by~\eqref{E:JSSmodJ} we have
\[
   H_1(\calS/\calJ[\Delta_{P,\sigma}])\ \cong\  H_0(\calJ[\Delta_{P,\sigma}])\,.
\]
  As there is only one vertex $v\in\Delta_{P,\sigma}$, we have
\[
  \calJ[\Delta_{P,\sigma}]\ =\ \bigoplus\limits_{\tau\in\Delta_{P,\sigma}} J(\tau)
  \ \xrightarrow{\ \partial_1\ }\  J(v)\,.
\]
 By definition, the map $\partial_1$ is surjective, since $J(v)=\sum_{v\in\tau} J(\tau)$, and if $v\in\Delta_{P,\sigma}$
 then so is any edge $\tau$ which contains $v$.
 It follows that $H_0(\calJ[\Delta_{P,\sigma}])=0$, and we are done.
\end{proof}
%%%%%%%%%%%%%%%%%%%%%%%%%%%%%%%%%%%%%%%%%%%%%%%%%%%%%%%%%%%%%%%%%%%%%%%%%%%%%%%%%

\begin{remark}
   Proposition~\ref{prop:FiniteLength} requires only Conditions (2) and (3) of Definition~\ref{D:generic}.\hfill$\diamond$
\end{remark}

\begin{remark}
   For rectilinear meshes, Schenck and Stillman show that $C^r(\Delta)$ is a free $S$-module if and only if
   $H_1(\calS/\calJ)=0$~\cite[Thm.~4.1]{ShSt97b}.
   This does not hold for semialgebraic splines.
   Even the spline module over a mesh with a single interior vertex is not necessarily free---this is why the local
   analysis for semialgebraic splines in~\cite{DiPSS} is more complicated than the local analysis in the rectilinear
   case. \hfill$\diamond$
\end{remark}

%%%%%%%%%%%%%%%%%%%%%%%%%%%%%%%%%%%%%%%%%%%%%%%%%%%%%%%%%%%%%%%%%%%%%%%%%%%%%%%%%
\begin{theorem}\label{T:GenericDimensionFormula}
  Suppose $\Delta$ is a generic mesh with $\phi_2$ faces and $\phi_1$ interior edges.
  For each interior vertex $v$, let $t_v$ be the minimum of $r{+}2$ and the number of arcs meeting at $v$ and put
  $a_v=\lfloor (r{+}1)/(t_v{-}1)\rfloor$.
  Then, for $d\gg 0$,
  \begin{equation}\label{E:GenericDimensionFormula}
   \dim C^r_d(\Delta)\ =\ 
   (\phi_2-\phi_1)\tbinom{d+2}{2} + \sum_{\tau\in\Delta^\circ_1} \tbinom{d-(r+1)n_\tau+2}{2}
   + \sum_{v\in\Delta^\circ_0}\left( \tbinom{r+a_v +2}{2}-t_v\tbinom{a_v+1}{2}\right)\,.
 \end{equation}
\end{theorem}
%%%%%%%%%%%%%%%%%%%%%%%%%%%%%%%%%%%%%%%%%%%%%%%%%%%%%%%%%%%%%%%%%%%%%%%%%%%%%%%%%

%%%%%%%%%%%%%%%%%%%%%%%%%%%%%%%%%%%%%%%%%%%%%%%%%%%%%%%%%%%%%%%%%%%%%%%%%%%%%%%%%
\begin{proof}
  The formula is obtained from~\eqref{Eq:dimFormula} by substituting the appropriate expressions for $\dim J(v)_d$ and
  $\dim H_0(\calJ)_d$.
  By Proposition~\ref{prop:FiniteLength}, $\dim H_0(\calJ)_d=0$ for $d\gg 0$, so this term drops out.
  It follows from~\cite[Thm.~4.1]{DiPSS} that the multiplicity of the scheme defined by $S/J(v)$ (using Conditions (1)
  and (2) of Definition~\ref{D:generic}) is the same as the multiplicity of the scheme defined by $S/I(v)$, where $I(v)$ is
  the ideal defined by $(r{+}1)$st powers of the (linear) initial forms of the edge forms meeting at $v$.
  The multiplicity of this scheme is $\tbinom{r+a_v +2}{2}-t_v\tbinom{a_v+1}{2}$ (see~\cite[Cor.~3.4]{DiPSS}).
  Thus the formula is proved.
\end{proof}
%%%%%%%%%%%%%%%%%%%%%%%%%%%%%%%%%%%%%%%%%%%%%%%%%%%%%%%%%%%%%%%%%%%%%%%%%%%%%%%%%

\begin{remark}
   Theorem~\ref{T:GenericDimensionFormula} requires Conditions (1), (2), and (3) of
   Definition~\ref{D:generic}.\hfill$\diamond$ 
\end{remark}

\begin{example}\label{E:Comparison}
Figure~\ref{F:ScSt_distinct} shows a the mesh $\Delta$ obtained by altering the mesh of Figure~\ref{F:cell_complex} so
that it satisfies Definition~\ref{D:generic} as follows.
Keep the central edge defined by $y=0$.
Let the top left edge be the line segment between $(-\frac{3}{2},1)$ and $(-1,0)$ with edge form $2x+y+2z$, and the edge to 
its right be the segment of the parabola $yz+(x+z)(x+2z)$ between $(-1,0)$ and $(0,1)$.
%%%%%%%%%%%%%%%%%%%%%%%%%%%%%%%%%%%%%%%%%%%%%%%%%%%%%%%%%%%%%%%%%%%%%%%%%%%%%%%%%
\begin{figure}[htb]
  \centering
   \includegraphics{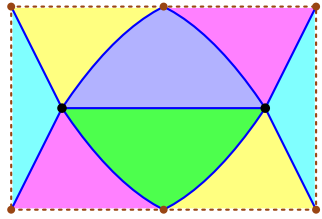}
   \caption{Altered mesh from Example~\ref{E:GenericityAssumptions}.}
   \label{F:ScSt_distinct}
\end{figure}
%%%%%%%%%%%%%%%%%%%%%%%%%%%%%%%%%%%%%%%%%%%%%%%%%%%%%%%%%%%%%%%%%%%%%%%%%%%%%%%%%
The other six non-horizontal edges are obtained from these by symmetry.
The first twelve columns of Table~\ref{Tbl:Comparison} record the Hilbert functions and the last the Hilbert polynomials of
$C^r(\Delta)$, computed using~\eqref{Eq:dimFormula}.
%%%%%%%%%%%%%%%%%%%%%%%%%%%%%%%%%%%%%%%%%%%%%%%%%%%%%%%%%%%%%%%%%%%%%%%%%%%%%%%%%
\begin{table}[htb]
\caption{Comparison of Hilbert function and polynomial for Example~\ref{E:Comparison}}\label{Tbl:Comparison}
\renewcommand{\arraystretch}{1.2}
\begin{tabular}{|c|cccccccccccc|c|}
\hline
$d$ & 0 & 1 & 2 & 3 & 4 & 5 & 6 & 7 & 8 & 9 & 10 & 11 & $\HP(C^r(\Delta),d)$ \\
\hline
$\dim C^0_d(\Delta)$ & 1 & \defcolor{{\bf 4}} & 15 & 34 & 61 & 96 & 139 & 190 & 249 & 316 & 391 & 474 & $4d^2-d+1$\\
\hline
$\dim C^1_d(\Delta)$  & 1 & 3 & \underline{6} & \defcolor{{\bf 11}} & 25 & 47 & 77 & 115 & 161 & 215 & 277 & 347 & $4d^2-14d+17$ \\
\hline
$\dim C^2_d(\Delta)$  & 1 & 3 & 6 & 10 & 15 & \underline{\defcolor{{\bf 23}}} & 38 & 63 & 96 & 137 & 186 & 243 & $4d^2-27d+56$ \\
\hline
$\dim C^3_d(\Delta)$  & 1 & 3 & 6 & 10 & 15 & 21 & 28 & \underline{\defcolor{{\bf 38}}} & 54 & 82 & 118 & 162 & $4d^2-40d+118$ \\
\hline
$\dim C^4_d(\Delta)$  & 1 & 3 & 6 & 10 & 15 & 21 & 28 & 36 & 45 & \defcolor{{\bf 58}} & \underline{77} & 106 & $4d^2-53d+205$ \\
\hline
\end{tabular}

\end{table}
%%%%%%%%%%%%%%%%%%%%%%%%%%%%%%%%%%%%%%%%%%%%%%%%%%%%%%%%%%%%%%%%%%%%%%%%%%%%%%%%%
As in Table~\ref{T:HilbertFunction}, the values at the postulation numbers are underlined and the first space having a 
nontrivial spline is \defcolor{{\bf highlighted}}.
(In the first row, $\dim C^0_d(\Delta)=\HP(C^0(\Delta),d)$ for every $d\ge 0$.)\hfill$\diamond$
\end{example}
%%%%%%%%%%%%%%%%%%%%%%%%%%%%%%%%%%%%%%%%%%%%%%%%%%%%%%%%%%%%%%%%%%%%%%%%%%%%%%%%%

The dimension of $C^r_d(\Delta)$ could be smaller or larger than the formula~\eqref{E:GenericDimensionFormula}
of Theorem~\ref{T:GenericDimensionFormula}, as is evident in Table~\ref{Tbl:Comparison}.
It is thus important to know how large $d$ should be to guarantee the exactness of this formula.
This is the content of Section~\ref{S:regularity}.

%%%%%%%%%%%%%%%%%%%%%%%%%%%%%%%%%%%%%%%%%%%%%%%%%%%%%%%%%%%%%%%%%%%%%%%%%%%%%%%%%
\section{Upper bound on regularity}\label{S:regularity}
We establish a bound on how large $d$ must be in order for the formula~\eqref{E:GenericDimensionFormula} of
Theorem~\ref{T:GenericDimensionFormula} to give the dimension of $C^r_d(\Delta)$ for a generic semialgebraic mesh. 
The closure of the union of the two faces adjacent to an interior edge $\tau$ is its \demph{star}, denoted
\defcolor{$\st(\tau)$}.

%%%%%%%%%%%%%%%%%%%%%%%%%%%%%%%%%%%%%%%%%%%%%%%%%%%%%%%%%%%%%%%%%%%%%%%%%%%%%%%%%
\begin{theorem}\label{T:GenericDimensionFormulaWithBound}
  Suppose that $\Delta$ is a generic mesh.
  Then the formula~\eqref{E:GenericDimensionFormula} for $\dim C^r_d(\Delta)$ holds for $d\ge D-2$, where
  $D=\max_{\tau\in\Delta^\circ_1}\{D_\tau\}$ and
\[
   D_\tau\ :=\ (r+1)\cdot\sum_{\tau'\in(\mbox{\footnotesize{\textup{st}}}(\tau))_1\cap\Delta^\circ_1} \deg(G_{\tau'})\,.
\]
\end{theorem}
%%%%%%%%%%%%%%%%%%%%%%%%%%%%%%%%%%%%%%%%%%%%%%%%%%%%%%%%%%%%%%%%%%%%%%%%%%%%%%%%%

\begin{example}
  Let $\Delta$ be the generic mesh of Example~\ref{E:Comparison}.
  The largest value for $D_\tau$ is $9(r{+}1)$ which occurs for the central edge $\tau$ connecting the interior vertices
  $(-1,0)$ and $(1,0)$.
  Thus the formulas $P^r(d)$ for $\dim C^r_d(\Delta)$ in Example~\ref{E:Comparison} hold for $d\ge 9(r{+}1){-}2$.
  From Table~\ref{Tbl:Comparison} the formula holds for much lower values of $d$.
  Thus the bound from Theorem~\ref{T:GenericDimensionFormulaWithBound} is typically larger than the actual postulation
  number. \hfill$\diamond$
\end{example}
%%%%%%%%%%%%%%%%%%%%%%%%%%%%%%%%%%%%%%%%%%%%%%%%%%%%%%%%%%%%%%%%%%%%%%%%%%%%%%%%%

Theorem~\ref{T:GenericDimensionFormulaWithBound} follows from arguments similar to those from~\cite{DiP_Mixed}.
We sketch this.
The postulation number of a graded $S$-module $M$ is bounded above in terms of
its Castelnuovo-Mumford \demph{regularity} \defcolor{$\reg(M)$}.
For the arguments given, we need only that the postulation number of $C^r(\Delta)$ is bounded by
$\reg(C^r(\Delta)){-}2$ (see~\cite[Lem.~4.6]{DiP_Mixed} and Remark~\ref{rem:pdim}).

Sections~4 and 5 of~\cite{DiP_Mixed} show if $\Delta$ is a rectilinear mesh then the regularity of $C^r(\Delta)$ is bounded
by the regularity of its submodules supported on the star of an edge.
The same arguments may be used to establish this result for semi-algebraic splines.
Theorem~\ref{T:GenericDimensionFormulaWithBound} follows by showing that the regularity of the submodule of $C^r(\Delta)$
supported on $\mbox{st}(\tau)$ is at most
$\left(\sum_{\tau'\in\Delta^\circ_1\cap\mbox{st}(\tau)} \deg(G_{\tau'})(r+1)\right){-}1$.
We give more details in the following subsections.
They only require that $\Delta$ satisfies Condition~(4) in Definition~\ref{D:generic}.

%%%%%%%%%%%%%%%%%%%%%%%%%%%%%%%%%%%%%%%%%%%%%%%%%%%%%%%%%%%%%%%%%%%%%%%%%%%%%%%%%
\begin{remark}\label{rem:pdim}
  Most results of~\cite{DiP_Mixed} (in particular~\cite[Lem.~4.6]{DiP_Mixed}) use that $C^r(\Delta)$ has
  {\sl projective dimension} at most one.
 Altering the map~\eqref{Eq:spline_kernel} in Lemma~\ref{L:SplineModule} shows that a spline module $C^r(\Delta)$
 for a semialgebraic mesh $\Delta$ is the kernel of a map between {\sl free} $S$-modules.
 This implies that the projective dimension of $C^r(\Delta)$ is at most one.
 Observe that $C^r(\Delta)$ is the kernel of the graded map
 \[
   S^{\Delta_2}\oplus \bigoplus\limits_{\tau\in\Delta^\circ_1} S(-(r+1)\deg(G_\tau))
   \ \longrightarrow\  S^{\Delta^\circ_1}\,,
 \]
 which is the map $\partial$ of~\eqref{Eq:spline_kernel} on the first summand and multiplication by
 $G_\tau^{r+1}$ on the summand corresponding to $\tau\in\Delta^\circ_1$.
 This is described in~\cite{BiRo91}; the only difference for semialgebraic meshes is to replace the linear forms by edge
 forms. 
\end{remark}
%%%%%%%%%%%%%%%%%%%%%%%%%%%%%%%%%%%%%%%%%%%%%%%%%%%%%%%%%%%%%%%%%%%%%%%%%%%%%%%%%

%%%%%%%%%%%%%%%%%%%%%%%%%%%%%%%%%%%%%%%%%%%%%%%%%%%%%%%%%%%%%%%%%%%%%%%%%%%%%%%%%
\subsection{First reduction}
For a face $\sigma$, write \defcolor{$C^r_\sigma(\Delta)$} for the submodule of splines of $C^r(\Delta)$ which are
supported on $\sigma$.
Likewise, for an interior edge $\tau$ write \defcolor{$\lsone{r}{\tau}{\Delta}$} for the submodule of splines of
$C^r(\Delta)$ which are supported on the star of $\tau$.
Define 
\[ 
  \defcolor{LS^{r}(\Delta)}\ :=\ \sum_{\tau\in\Delta^\circ_1} \lsone{r}{\tau}{\Delta}\,.
\]
We require Condition~(4) of Definition~\ref{D:generic} to guarantee that if $P$ is a prime of codimension $1$ (so that
$P=\langle f\rangle$ for some polynomial $f$), then $\overline{\Delta_{P,\sigma}}$ is either
$\sigma$ or, if $P=\langle G_\tau\rangle$ for some $\tau\subset\sigma$, then $\Delta_{P,\sigma}=\mbox{st}(\tau)$.
Hence $LS^{r}(\Delta)$ is generated by splines which are supported on
sets of the form $\Delta_{P,\sigma}$ where $\codim(P)=1$. 

%%%%%%%%%%%%%%%%%%%%%%%%%%%%%%%%%%%%%%%%%%%%%%%%%%%%%%%%%%%%%%%%%%%%%%%%%%%%%%%%%
\begin{lemma}\label{L:locsup}
  Let $K$ be the cokernel of the inclusion of $LS^{r}(\Delta)$ into $C^r(\Delta)$.
  Then $K$ is supported at primes of codimension at least two.
\end{lemma}
\begin{proof}
The point is that upon localization at primes of codimension $\le 2$, the inclusion of $LS^{r}(\Delta)$ into
$C^r(\Delta)$ becomes an isomorphism.
This is verified using Lemma~\ref{L:SplineModule}.
\end{proof}

\begin{theorem}\label{T:FirstReduction}
If $\Delta$ is a generic mesh, then $\reg(C^r(\Delta))\le \reg(LS^{r}(\Delta))$.
\end{theorem}
\begin{proof}
  See~\cite[Thm.~4.7]{DiP_Mixed}.
  This is deduced from Lemma~\ref{L:locsup} using~\cite[Prop.~A.7]{DiP_Mixed} (we also need the projective dimension of $C^r(\Delta)$ to be bounded above by one---see Remark~\ref{rem:pdim}).
\end{proof}

%%%%%%%%%%%%%%%%%%%%%%%%%%%%%%%%%%%%%%%%%%%%%%%%%%%%%%%%%%%%%%%%%%%%%%%%%%%%%%%%%
\subsection{Second reduction}

We explain how the following result, which is~\cite[Thm.~5.5]{DiP_Mixed}, can be derived in our context of a mesh $\Delta$
which satisfies Condition (4) of Definition~\ref{D:generic}.

%%%%%%%%%%%%%%%%%%%%%%%%%%%%%%%%%%%%%%%%%%%%%%%%%%%%%%%%%%%%%%%%%%%%%%%%%%%%%%%%%
\begin{theorem}\label{T:SecondReduction}
  The regularity of $LS^{r}(\Delta)$ is at most the maximum of the regularity of the modules
  $\lszero{r}{\sigma}{\Delta}$ for $\sigma$ a face of $\Delta$ and $\lsone{r}{\tau}{\Delta}$ for $\tau$ an interior edge of
  $\Delta$.
\end{theorem}
%%%%%%%%%%%%%%%%%%%%%%%%%%%%%%%%%%%%%%%%%%%%%%%%%%%%%%%%%%%%%%%%%%%%%%%%%%%%%%%%%

The proof of Theorem~5.5 from~\cite{DiP_Mixed} carries over verbatim, as it is topological
rather than algebraic.
In~\cite[Sect.~5]{DiP_Mixed} the first author constructs a chain complex of the form 
\[
  0\ \rightarrow\ C_n\ \rightarrow C_{n-1}\ \rightarrow\ \dotsb\ \rightarrow\
  C_1\ \rightarrow\ LS^{r}(\Delta)\ \rightarrow\ 0\,,
\]
where $C_1$ is the direct sum of $\lsone{r}{\tau}{\Delta}$ for $\tau$ an interior edge and the remaining $C_k$
are direct sums of $\lszero{r}{\sigma}{\Delta}$ for certain faces $\sigma$.
This chain complex is shown to be exact using arguments that depend only on considerations of support, so exactness holds
in our context.
Given exactness, the bound of Theorem~\ref{T:SecondReduction} follows from the behavior of regularity in chain complexes
(see~\cite[Lem.~A.4]{DiP_Mixed}). 

%%%%%%%%%%%%%%%%%%%%%%%%%%%%%%%%%%%%%%%%%%%%%%%%%%%%%%%%%%%%%%%%%%%%%%%%%%%%%%%%%
\subsection{Third reduction}
We follow the arguments of~\cite[Sect.~6]{DiP_Mixed} to establish an upper bound on $\reg(\lsone{r}{\tau}{\Delta})$.
Notice that we may write $F\in C^r(\Delta)$ as $F=\sum_{\sigma} F_\sigma e_\sigma$, where
$F_\sigma=F|_{\sigma}$ is the restriction of $F$ to the face $\sigma$ and $e_\sigma$ is a formal basis
symbol.

For an interior edge $\tau$, let $\sigma,\sigma'$ be its incident faces with corresponding basis symbols $e,e'$.
Consider the three splines in $\lsone{r}{\tau}{\Delta}$,
\[
   F_1\ :=\ \prod_{\tau'\in\sigma} G_{\tau'}^{r+1}\cdot e\, ,\  
  F_2\ :=\ \prod_{\tau'\in\sigma'} G_{\tau'}^{r+1}\cdot e'\, , \ \mbox{ and }\ 
  F_\tau\ := \ \prod_{\tau'\in\st(\tau),\tau'\neq\tau}\!\! G_{\tau'}^{r+1}\cdot (e+e')\, .
\]
(In each product $\tau'$ runs over interior edges.)
As in~\cite[Sect.~6]{DiP_Mixed}, the quotient of $\lsone{r}{\tau}{\Delta}$ by the submodule $N_\tau$ generated by
$F_1,F_2,$ and $F_\tau$ has codimension $\ge 2$, so $\reg(\lsone{r}{\tau}{\Delta})$ is bounded above by the regularity of
$N_\tau$ (see~\cite[Prop.~A.7]{DiP_Mixed}; we also need that the projective dimension of $\lsone{r}{\tau}{\Delta}$ is at
most one---see Remark~\ref{rem:pdim}).  

We establish the regularity of $N_\tau$.  The splines $F_1,F_2,$ and $F_\tau$ have a single syzygy 
\[
  \Bigl(\prod\limits_{\tau'\in\sigma',\tau'\neq\tau} G_{\tau'}\Bigr) F_1\ +\
  \Bigl(\prod\limits_{\tau'\in\sigma,\tau'\neq\tau} G_{\tau'}\Bigr)F_2\ -\ G_\tau F_\tau\ =\ 0\,,
\]
which has degree $D_\tau=\sum_{\tau'\in\mbox{\footnotesize{st}}(\tau)\cap\Delta^\circ_1} \deg(G_{\tau'})(r+1)$.
Thus the regularity of the submodule $N_\tau$ of $\lsone{r}{\tau}{\Delta}$ generated by these three splines is $D_\tau{-}1$.
Theorem~\ref{T:GenericDimensionFormulaWithBound} follows from Theorems~\ref{T:FirstReduction} and~\ref{T:SecondReduction}.

%%%%%%%%%%%%%%%%%%%%%%%%%%%%%%%%%%%%%%%%%%%%%%%%%%%%%%%%%%%%%%%%%%%%%%%%%%%%%%%%%
\section{Concluding Remarks}\label{S:ConcludingRemarks}

There are many meshes for which neither Theorem~\ref{T:NSubdivisionStructure} nor
Theorem~\ref{T:GenericDimensionFormula} apply but the Hilbert polynomial can still be computed using homological algebra.
For these it is useful to have a presentation for the homology module $H_1(\calS/\calJ)$  derived by Schenck and Stillman
in~\cite[Lem.~3.8]{ShSt97b}.
This presentation holds for semialgebraic meshes after replacing the linear forms $l_\tau$ with the edge forms
$G_\tau$.
This presents $H_0(\calJ)$ as a quotient of the free module $S^{\Delta_1^\circ}$ (with the grading of each summand
shifted to reflect the degree of the edge form $G_\tau^{r+1}$) on the interior edges by the submodule generated by syzygies
of the ideals $J(v)$ (that is, the relations among collections of forms $G_\tau^{r+1}$ with $v\in\tau$) along with
generators of $S^{\Delta_1^\circ}$ that correspond to edges meeting the boundary of $\Delta$.
To illustrate this, we verify the Hilbert polynomials in Table~\ref{T:HilbertFunction} for $C^r(\Delta)$ where
$\Delta$ is the mesh of Figure~\ref{F:cell_complex}. 

%%%%%%%%%%%%%%%%%%%%%%%%%%%%%%%%%%%%%%%%%%%%%%%%%%%%%%%%%%%%%%%%%%%%%%%%%%%%%%%%%
\begin{example}\label{E:FirstMesh}
  Let $\Delta$ be the mesh of Figure~\ref{F:cell_complex}.
  Let $f$ and $g$ be the forms defining the two parabolas.
  The Schenck-Stillman presentation for $H_0(\calJ)$ shows that it is a cyclic $S$-module whose generator corresponds to
  the interior edge of $\Delta$ between $v=(-1,0)$ and $v'=(1,0)$.
  We have 
\[
  J(v)\ =\ J(v')\ =\ \langle f^{r+1},g^{r+1},y^{r+1}\rangle\,.
\]
 The presentation shows that $H_0(\calJ)$ is the quotient of $S(-r{-}1)$ by the ideal generated by coefficients $H$ of
 $y^{r+1}$ in all possible relations of the form 
\[
   Ff^{r+1}\ +\ Gg^{r+1}\ +\ Hy^{r+1}\ =\ 0\,.
\]
This is the \demph{colon} ideal, \defcolor{$\langle f^{r+1},g^{r+1}\rangle: y^{r+1}$}.
(For ideals $I$ and $J$, the colon ideal $I:J$ is the ideal of all ring elements multiplying $J$ into $I$.)
Hence
\[
   H_1(\calS/\calJ)\ \cong\ H_0(\calJ)\ \cong\  \langle f^{r+1},g^{r+1}\rangle:y^{r+1}\,.
\]
 This can be analyzed using the graded short exact sequence
 \begin{equation}\label{E:MultiplicationSequence}
  0\ \rightarrow\ \frac{S(-r-1)}{\langle f^{r+1},g^{r+1}\rangle:y^{r+1}}\ \xrightarrow{\ \cdot y^{r+1}\ }\
  \frac{S}{\langle f^{r+1},g^{r+1}\rangle}\ \rightarrow\
   \frac{S}{\langle f^{r+1},g^{r+1},y^{r+1}\rangle}\ \rightarrow\ 0\,.
 \end{equation}
  The last quotient is $S/J(v)=S/J(v')$.  From~\eqref{E:MultiplicationSequence} we obtain
 \begin{equation}\label{E:HomologyHilbert}
   \HP(H_0(\calJ),d)\ =\ (2r+2)^2-\HP(S/J(v),d)\,.
 \end{equation}
Since the Hilbert polynomials of $H_0(\calJ)$ and $S/J(v)$ are both eventually constant~\eqref{E:HomologyHilbert} relates
the multiplicities of $H_0(\calJ)$ and $S/J(v)$. 

By~\eqref{E:HomologyHilbert} it suffices to determine the multiplicity of $S/J(v)$.
Notice that
\[
  \sqrt{J(v)}\ =\ \sqrt{\langle f,g,y \rangle}\ =\ \sqrt{\langle y,x^2-z^2 \rangle}\ =\
  \langle x-z,y\rangle\cap \langle x+z,y\rangle
\]
 is supported at the two interior vertices $v$ and $v'$.
 So $J(v)$ fails Condition (2) of Definition~\ref{D:generic} and in particular we cannot use~\cite[Thm.~4.1]{DiPSS} to
 evaluate the multiplicity of $S/J(v)$.
 However, following~\cite[Rem.~6.3]{DiPSS}, we can use~\cite[Thm.~4.1]{DiPSS} {\sl locally} at each of the points $v$
 and $v'$ since the tangents of the curves defined by $f,g,$ and $y$ are distinct at both points.
 Thus the multiplicity of $S/J(v)$ is the sum of the multiplicities of $S/J(v)$ at $v$ and at $v'$.
 These multiplicities are the same as the multiplicity of $S/\langle y^{r+1},(y{-}z)^{r+1},(y{+}z)^{r+1}\rangle$.
  By~\cite[Cor.~3.4]{DiPSS}, this ideal has multiplicity $\binom{a+r+2}{2}-3\binom{a+1}{2}$, where
  $a=\lfloor(r+1)/2\rfloor$.
  Hence the multiplicity of $S/J(v)$ is twice this multiplicity.
  Using Corollary~\ref{C:DimFormula} and~\eqref{E:HomologyHilbert}, we get
 \begin{multline*}
  \quad  \HP(C^r(\Delta),d)\ =\ \binom{d+2}{2}+\binom{d+1-r}{2}+8\binom{d-2r}{2}\\
   +(2r+2)^2+2\left(\binom{a+r+2}{2}-3\binom{a+1}{2}\right)\,. \quad
 \end{multline*}
Considering $r$ even and odd, this becomes
\[
    \HP(C^r(\Delta),d)\ =\ 
   \left\lbrace
    \begin{array}{ll}
      4d^2-(34k-12)d+(88k^2-51k+8) & \mbox{if } r=2k-1\\\vspace{2pt}
     4d^2-(34k+5)d+(88k^2+37k+5) & \mbox{if } r=2k
    \end{array}
     \right.\ .
\]
This verifies the Hilbert polynomials in Table~\ref{T:HilbertFunction}.\hfill$\diamond$

\end{example}

\begin{remark}
   The underlined postulation numbers in Table~\ref{T:HilbertFunction} behave roughly as $4r+1$, which is what is
   expected if the regularity of $H_0(\calJ)$ is $4r+2$.
   From the short exact sequence~\eqref{E:MultiplicationSequence} we indeed expect $\reg(H_0(\calJ))\le 4r+2$, since $4r+2$
   is the regularity of the complete intersection $S/\langle f^{r+1},g^{r+1}\rangle$.  \hfill$\diamond$
\end{remark}

\begin{remark}
This presentation for $H_1(\calS/\calJ)$ can also be used to analyze semialgebraic splines on meshes which we shall
call quasi cross-cut meshes.
These are meshes in which every interior edge is part of an arc of a single algebraic curve that meets the boundary of the
mesh (this arc may pass through several interior vertices).
As Schenck and Stillman observed~\cite{ShSt97a,ShSt97b} for rectilinear meshes, the presentation for
$H_1(\calS/\calJ)$ shows that $H_1(\calS/\calJ)=0$.
Hence such splines satisfy dimension formulas extending those of Chui and Wang~\cite{ChuiWang83} and
Schumaker~\cite{Schum84} for rectilinear quasi cross-cut meshes.
The caveat is that the quotients $S/J(v)$ at vertices may still be quite difficult to analyze. \hfill$\diamond$
\end{remark}

\begin{remark}
   As in Example~\ref{E:FirstMesh}, it is possible to analyze meshes for which $H_0(\calJ)$ does not have finite length.
   In this case there is a contribution to $\HP(C^r(\Delta),d)$ coming from this homology module.
   If we relax Condition (3) of Definition~\ref{D:generic} to allow $H_0(\calJ)$ to be supported at points and analyze the
   contributions to $\HP(C^r(\Delta),d)$ coming from $H_0(\calJ)$ one should get a formula generalizing that of
   McDonald and Schenck in~\cite{McSh09}. \hfill$\diamond$
\end{remark}

\begin{remark}
  Condition (1) of Definition~\ref{D:generic} is required for the approach from~\cite{DiPSS}.
  By~\cite[Ex.~6.1]{DiPSS} this fails for large $r$ if two edge forms are tangent.
  In practice (see~\cite{Davydov16,DS17,DY17}) one may wish to impose vanishing conditions across a $C^1$ piecewise
  polynomial boundary, which requires such tangency at boundary vertices.
  Even in this situation it is likely possible to work out dimension formulas for small values of $r$ using our
  approach. \hfill$\diamond$ 
\end{remark}

%%%%%%%%%%%%%%%%%%%%%%%%%%%%%%%%%%%%%%%%%%%%%%%%%%%%%%%%%%%%%%%%%%%%%%%%%%%%%%%%%
\providecommand{\bysame}{\leavevmode\hbox to3em{\hrulefill}\thinspace}
\providecommand{\MR}{\relax\ifhmode\unskip\space\fi MR }
\providecommand{\MRhref}[2]{%
  \href{http://www.ams.org/mathscinet-getitem?mr=#1}{#2}
}
\providecommand{\href}[2]{#2}

\end{document}